\documentclass[a4paper,12pt]{amsart}
\usepackage[utf8]{inputenc}
\usepackage{appendix}
\usepackage{booktabs}
\usepackage{caption}
\usepackage[T1]{fontenc}
\usepackage[UKenglish]{babel}
\usepackage[margin=18mm]{geometry}
\usepackage{color}

\usepackage{graphicx}

\usepackage{amsmath,amssymb,amsfonts,amsthm}
\usepackage{mathrsfs,eucal,dsfont}
\usepackage{verbatim,enumitem}

\usepackage{hyperref,url}

\newcommand{\R}{\mathds R}

\newcommand{\N}{\mathds N}

\newcommand{\I}{\mathds 1}

\def\aa{\alpha}
\def\dd{\delta}
\def\d{{\rm d}}
\def\<{\langle}
\def\>{\rangle}

 \def\tt{\tilde}
 \def\ff{\frac}
 \def\ss{\sqrt}
\def\bb{\beta}

\def\R{\mathbb R}  \def\ff{\frac} \def\ss{\sqrt} 
\def\N{\mathbb N} \def\kk{\kappa} 
\def\dd{\delta}  \def\vv{\varepsilon} 
\def\<{\langle} \def\>{\rangle}  
  \def\nn{\nabla}  
\def\d{\text{\rm{d}}} \def\bb{\beta} \def\aa{\alpha} 
  \def\si{\sigma} 
 \def\beq{\begin{equation}}  
 
\def\e{\text{\rm{e}}}  \def\OO{\Omega}  
 \def\tt{\tilde} 
 \def\P{\mathbb P}

  \def\ll{\lambda}
 
\def\E{\mathbb E}

\def\to{\rightarrow}
\def\8{\infty}\def\3{\triangle}
\def\W{\mathbb{W}}\def\1{\lesssim}

\renewcommand{\bar}{\overline}
\renewcommand{\hat}{\widehat}
\renewcommand{\tilde}{\widetilde}

\newtheorem{theorem}{Theorem}[section]
\newtheorem{lemma}[theorem]{Lemma}

\theoremstyle{definition}

\newtheorem{remark}[theorem]{Remark}

\numberwithin{equation}{section}

\begin{document}

\title[$L^2$-Wasserstein   contraction of modified Euler schemes] {$L^2$-Wasserstein   contraction of modified Euler schemes for SDEs with high diffusivity and applications}

\author{
Jianhai Bao\qquad
Jiaqing Hao}
\date{}
\thanks{\emph{J.\ Bao:} Center for Applied Mathematics, Tianjin University, 300072  Tianjin, P.R. China. \url{jianhaibao@tju.edu.cn}}

\thanks{\emph{J.\ Hao:}
Center for Applied Mathematics, Tianjin University, 300072  Tianjin, P.R. China. \url{hjq_0227@tju.edu.cn}}

\allowdisplaybreaks

\maketitle

\begin{abstract}
 In this paper,
 we are concerned with a modified  Euler scheme for the   SDE   under consideration, where the drift is  of super-linear growth and dissipative merely outside a closed ball.
 By adopting the synchronous coupling, along with the construction  of an   equivalent
 quasi-metric, the $L^2$-Wasserstein contraction  of the modified Euler scheme is addressed provided that the diffusivity is large enough. In particular, as a direct  application, the $L^2$-Wasserstein contraction  of the projected (truncated) Euler scheme and the tamed Euler algorithm is explored under much more explicit conditions imposed on drifts.
The theory derived on the
 $L^2$-Wasserstein contraction  related to  the modified EM scheme has numerous applications. In addition to applications on Poincar\'{e} inequalities (with respect to   the numerical  transition kernel and the numerical invariant probability measure), concentration inequalities for empirical averages, and non-asymptotic convergence  bounds in  the KL-divergence, in this paper we present another two potential applications. One concerns the  non-asymptotic $L^2$-Wasserstein bounds
associated with  the classical Euler scheme, the projected Euler scheme and the tamed Euler recursion, respectively, which further implies the $L^2$-Wasserstein  error bounds between the exact invariant probability measures and the numerical counterparts.
  It is worthy to emphasize that the associated convergence rate is improved greatly  in contrast to the existing literature. Another application is devoted to the  strong law of large numbers of additive functionals corresponding  to the modified Euler algorithm, where the observable functions involved  are allowed to be of polynomial growth,
 and the associated convergence rate is also enhanced remarkably. Particularly,
 the strong law of large numbers for the classical Euler scheme, the projected Euler scheme, and the tamed Euler recursion are treated  simultaneously.

\medskip


\smallskip

\end{abstract}
\section{Background, main results  and applications}
\subsection{Background}
Over the past decade, via the probabilistic approaches (e.g., coupling methods, Harris' theorems, and functional inequalities), the
  ergodicity of SDEs   has  advanced considerably in various scenarios. As far as SDEs under consideration are concerned, where the driven noises are Brownian motions,
 we refer to e.g.  \cite{Eberle16} for non-degenerate SDEs, \cite{EGZ} concerned with  stochastic Hamiltonian systems, \cite{Schuh} regarding kinetic Langevin dynamics with distribution-dependent forces, and \cite{Wang23} with regard to non-dissipative McKean-Vlasov SDEs, to name just a few. Meanwhile, in the past few years, the ergodicity of SDEs driven by pure jump processes  has also achieved  some  progresses. In \cite{LW},
 the so-called refined basic coupling approach was  proposed creatively to investigate the ergodicity for L\'{e}vy-driven SDEs, where the L\'{e}vy noise need not to be rotationally invariant.  Subsequently, the refined basic coupling method has been applied successfully to handle ergodicity in other setups; see e.g. \cite{BW,LMW,LWZ} for
 stochastic Hamiltonian systems and McKean-Vlasov SDEs driven by  L\'{e}vy noise.
 In the aforementioned literature, the ergodicity for the underlying stochastic systems was established under the $L^1$-Wasserstein distance, or the additive type Wasserstein distance, or the multiplicative type quasi-Wasserstein distance. In particular, via a reflection coupling trick, the weak  $L^1$-Wasserstein contraction was tackled in \cite{Eberle16} for SDEs, where the drifts are dissipative in the long distance. Later on,
 the $L^p$-Wasserstein ($p>1$) decay (rather than the weak $L^p$-Wasserstein contraction) was explored in \cite{LWb} for SDEs with partially dissipative drifts.  Among ergodicity under the $L^p$-Wasserstein distance for SDEs with non-uniformly dissipative drifts,
 the weak $L^2$-Wasserstein contraction plays a distinctive role since it links closely to, for instance,  functional inequalities and the KL-divergence as mentioned in
\cite{LMM,PM}. Recently, by means of the synchronous coupling,
 the weak $L^2$-Wasserstein contraction was  discussed  in \cite{LMM}   at the price of high diffusivity. In addition,
  the $L^2$-Wasserstein convergence to the equilibrium (which  is indeed a special $L^2$-Wasserstein contraction) was studied in \cite{Wang20} in case   the unique invariant probability measure (IPM for short) of the SDE under consideration satisfies the log-Sobolev inequality.

As one of the challenges encountered in statistics, machine learning, and data science, sampling from a target distribution with known parameters is essential. Particularly,  a novel sampling method for known distributions and a new algorithm for diffusion generative models were introduced in \cite{Zhang}. Another  potential way   for generating  samples from known distributions (i.e., the target distributions) is to manipulate an appropriate numerical algorithm to discretize (in time) an associated SDE, where the   unique IPM is the target distribution. Nowadays,
with the rapid development  on the ergodicity of continuous-time stochastic systems (as mentioned previously in the previous paragraph), the long-time analysis of  stochastic algorithms   has also gained   much more attention (see e.g. \cite{DD,DEEGM,MSH,Suzuki}, to list only a few) and moreover built  theoretical evidences for sampling from a target  distribution. Lately,  the issue on non-asymptotic bounds under the (weighted) total variation distance and the (quasi-)Wasserstein distance
 for sampling algorithms in the non-convex setting is brought into sharp focus; see e.g. \cite{BDMS,DAS,MMS,NMZ} and references therein.

 Specially,    lower bounds on contraction rates for Markov chains on general state spaces were derived in \cite{EM}, and then, as an application, the   $L^1$-Wasserstein contraction was investigated for the  Euler discretization of SDEs with non-globally contractive drifts. Soon afterwards, \cite{EM} was extended in \cite{HMW} to SDEs with general noise  (including    Brownian motions and  rotationally invariant stable processes). Recently, via a carefully tailored distance function and an appropriate coupling, the $L^1$-Wasserstein contraction  was investigated in \cite{SW} for three kinetic Langevin samplers with non-convex potentials.

 As stated in \cite{LMM}, the  $L^2$-Wasserstein contraction for stochastic algorithms  is vital due to the fact that it relates closely to the Poincar\'{e}  inequalities, concentration inequalities, entropy-cost regularization inequalities and non-asymptotic convergence bounds in the KL-divergence and the total variation. In particular, the  $L^2$-Wasserstein contraction for Euler schemes of elliptic diffusions and interacting particle systems was considered in \cite{LMM}. Recall that the classical Euler scheme is unstable once the drifts of SDEs under investigation are of super-linear growth; see, for example, \cite{HJK,MSH}. So,
no matter what \cite{EM}, \cite{HMW}, \cite{LMM}, and \cite{SW},  the Euler scheme was applied to discretize the underlying SDEs, where
 the  drifts  are supposed to be Lipschitz continuous (so it is of linear growth at most). Undoubtedly, such a strong restriction on  the growth of drifts
reduces applicability of the theory found in \cite{EM,HMW,LMM,SW}.

Admittedly, in the past few years,   theoretical foundations for non-asymptotic bounds under the $L^1$-Wasserstein distance and the (weighted) total variation distance for sampling algorithms are relatively well-established. Meanwhile, the long-time analysis of the tamed type numerical schemes
associated with SDEs with super-linearity has advanced greatly.  Specially,  most of the literature focuses on the non-asymptotic convergence bounds, which are based on the weak contraction of the underlying exact solutions;  see e.g.  \cite{BDMS,DAS,NMZ}. Nevertheless, the topic on the weak contraction under the Wasserstein distance of the associated tamed schemes is left open  in  \cite{BDMS,DAS,NMZ}.
Inspired by the references mentioned previously,  in this paper we shall go     beyond \cite{BDMS,DAS,LMM,NMZ}  and aim to  address the weak $L^2$-Wasserstein contraction for a modified Euler scheme (including the truncated/projected Euler scheme and the tamed Euler recursion as typical candidates) associated with a range of SDEs with super-linear drifts. Most importantly,  with the help of the theory concerning
the weak $L^2$-Wasserstein contraction,
we attempt to pave the way for tackling the  non-asymptotic $L^2$-Wasserstein   bounds
for Langevin sampling algorithms in a non-convex setting, where the potential involved  might be   of super-linear growth. Additionally,  the  improvement   on the $L^2$-Wasserstein convergence rate  with respect to the step size is another goal we want to seek. The above   accounts
can be regarded as the principal inspirations  impelling us to carry out the present  topic.

\subsection{Weak $L^2$-Wasserstein contraction}\label{sub1}
More precisely, in this paper we work on   an SDE in the form:
\begin{align}\label{E1}
\d X_t=b(X_t)\,\d t+\si\d W_t,
\end{align}
where  $\si\in\R$,   $(W_t)_{t\ge0}$ is a  $d$-dimensional Brownian motion  on a probability basis (i.e., a complete filtered probability space) $(\OO,\mathscr F,(\mathscr F_t)_{t\ge0},\P)$, and
$b:\R^d\to\R^d$ is measurable satisfying that
\begin{itemize}
 \item[(${\bf A}_0$)]
 there exist constants $L_0 >0$ and $\ell_0\ge0$ such that
 \begin{align*}
 |b(x)-b(y)|\le L_0\big(1+|x|^{\ell_0}+|y|^{\ell_0}\big)	|x-y|, \quad x,y\in\R^d.
 \end{align*}
 \end{itemize}
 Under $({\bf A}_0)$, the SDE \eqref{E1}  has a unique maximal local solution up to  life time.
 Throughout the paper, we shall assume that the SDE \eqref{E1} is strongly well-posed under appropriate conditions (e.g., (${\bf A}_0$) along
 with  a Lyapunov condition).

As we know, when the drift term $b$ is of linear growth, one of the simple ways of approximating the SDE \eqref{E1} is
the Euler-Maruyama (EM for short) scheme: for $\delta>0$ and integer $n\ge0,$
\begin{align}\label{WW}
X_{(n+1)\delta}^\delta=X_{ n \delta}^\delta+b(X_{n\delta}^\delta)+\si\triangle W_{n\delta},
\end{align}
where  $\triangle W_{n\delta}:=W_{(n+1)\delta}-W_{n\delta}$ (i.e.,  the increment of $(W_t)_{t\ge0}$ on the interval $[(n+1)\delta,n\delta]$). Nevertheless, the EM scheme \eqref{WW} no longer  works in case   $b$ is of super-linear growth. From this viewpoint, other numerical approximation candidates  have been proposed to tackle the difficulty arising from  the super-linearity of drifts. As a typical  candidate, the   tamed EM (TEM for abbreviation) scheme (see e.g. \cite{BDMS,HJK}): for $\delta>0,$ $n\ge0$ and $\gamma\in(0,1/2),$
\begin{align}\label{EW0-}
X_{(n+1)\delta}^{\delta,\gamma }= X_{n\delta}^{\delta,\gamma } + b^{(\delta)}(X_{n\delta}^{\delta,\gamma })\delta +\si\triangle W_{n\delta}
\end{align}
can be applied to discretize the SDE \eqref{E1}, where for $\ell_0\ge0 $ given in (${\bf A}_0$),
\begin{align}\label{EY-}
 b^{(\delta)}(x)=\frac{b(x) }{1+\delta^\gamma| x |^{\ell_0}},\quad x\in\R^d.
\end{align}
 On the other hand, inspired by e.g. \cite{BIK,Mao}, the SDE \eqref{E1} can also be approximated by the projected EM (PEM for short) scheme:  for $\delta>0,$ $n\ge0$ and $\gamma\in(0,1/2),$
\begin{align}\label{EWW}
X_{(n+1)\delta}^{\delta,\gamma}= \pi^{(\delta)}(X_{n\delta}^{\delta,\gamma}) +b ( \pi^{(\delta)}(X_{n\delta}^{\delta,\gamma}) )\delta+\si\triangle W_{n\delta}.
\end{align}
Herein, the truncated map $\pi^{(\delta)}$ is defined as below:
\begin{align}\label{ET-}
\pi^{(\delta)}(x)=\frac{1}{|x|}(|x|\wedge\varphi^{-1}(\delta^{-\gamma}))x\I_{\{|x|\neq 0\}}, \quad x\in\R^d,
\end{align}
where $[0,\8)\ni r\mapsto \varphi^{-1}(r)$ denotes the inverse function of  $\varphi(r):=1+r^{\ell_0}, r\ge0$.

In the present work,
aiming at fitting the schemes  \eqref{WW}, \eqref{EW0-} and \eqref{EWW} into a unified framework, we   put forward the following modified Euler recursion associated with \eqref{E1}: for   $\delta>0$ and  $n\ge0,$
\begin{align}\label{E2}
X_{(n+1)\delta}^\delta=\pi^{(\delta)}(X_{n\delta}^\delta)+b^{(\delta)}(\pi^{(\delta)}(X_{n\delta}^\delta))\delta+\si\triangle W_{n\delta},
\end{align}
  where $ b^{(\dd)}(x)$ and $\pi^{(\dd)}(x)$ are modified   versions of $b(x)$ and $ x$, respectively, (i.e.,  for each fixed $x\in\R^d$, $|b^{(\delta)}(x)-b(x)|\to0$ and $\big|\pi^{(\dd)}(x)-x\big|\to 0$ as $\delta\to0$), and
\begin{enumerate}
\item[$({\bf A}_1)$]the mapping $\pi^{(\delta)}:\R^d\to\R^d$ with $\pi^{(\delta)}({\bf0})={\bf0}$
is contractive, i.e., $$|\pi^{(\delta)}(x)-\pi^{(\delta)}(y)|\le |x-y|, \quad x,y\in\R^d.$$
\end{enumerate}

Below, let  ${\rm Id}$   be the identity map, i.e., ${\rm Id}(x)=x,x\in\R^d $.
Obviously,
by taking (i) $\pi^{(\delta)} ={\rm Id}$  and $b^{(\delta)} =b $, (ii) $\pi^{(\delta)} ={\rm Id}$ and $b^{(\delta)} $ introduced in \eqref{EY-}, as well as  (iii) $b^{(\delta)} =b $
and $\pi^{(\delta)} $ defined in \eqref{ET-}, the iteration  \eqref{E2} goes  back to the schemes \eqref{WW}, \eqref{EW0-} and \eqref{EWW}, respectively.  For other variants of \eqref{E2}, we would like to allude to e.g.  \cite{BDMS,DD,DEEGM}. In detail, \cite{BDMS} provided a
prototype of the   recursion \eqref{E2}   with $\pi^{(\delta)} ={\rm Id}$.  Hereafter,   \cite{DEEGM} investigated the   successive iteration of   functional autoregressive processes with the isotropic Gaussian noise of the form:
\begin{align*}
  X_{(n+1)\delta}^\delta=b^{(\delta)}(X_{n\delta}^\delta)+\si \ss{\delta}Z_{(n+1)\delta},
\end{align*}
where $b^{(\delta)}:\R^d\to\R^d$ is continuous satisfying \cite[({\bf H1})]{DEEGM}, and  $(Z_{n\delta})_{n\ge1}$ is a sequence of i.i.d. $d$-dimensional standard Gaussian random
variables. As a special case,  in \cite{DEEGM} the authors   considered  the EM scheme  with $b^{(\delta)}(x):=x+b(x)\delta$  for a   {\it  Lipschitz continuous} $b:\R^d\to\R^d$. Furthermore,
\cite{DD} tackled     a much more general version of \eqref{E2}; see in particular \cite[(19)]{DD} for more details.

As we know very well, the algorithms \eqref{WW}, \eqref{EW0-} and \eqref{EWW} are stable (i.e., they have finite moment in a finite horizon as the associated  exact solutions do)
 under certain conditions, respectively; see, for instance, \cite{BIK,HJK} for related details. Indeed, under some verifiable conditions (see in particular (${\bf H}_0)$ in Section \ref{sec2}), we can also  show that the scheme \eqref{E2} possesses  finite second-order moment in a finite horizon
 so it can preserve the corresponding stability; see Lemma \ref{nonexplosition} in Section \ref{sec2} for   details.

   To proceed, we introduce some  notations.
Denote  $\mathscr L_\xi$ by  the distribution of the random variable $\xi,$ and   $\mathscr P(\R^d)$ (resp. $\mathscr P_2(\R^d)$) refers to  the space of probability measures (resp. with finite second moment) on $\R^d.$
If   $\mathscr L_{X_0^\delta}=\mu\in\mathscr P(\R^d)$,
 we  write  $(X_{n\delta}^{\delta,\mu})_{n\ge0}$ instead of $(X^\dd_{n\delta})_{n\ge0}$ determined by \eqref{E2} to stress  the  initial distribution $\mu$. On some occasions,  in case of $\mu=\delta_x$ (i.e., Dirac's delta measure  centered at the point $x\in\R^d$), we  write $(X_{n\delta}^{\delta,x})_{n\ge0}$   in lieu of
 $(X_{n\delta}^{\delta,\delta_x})_{n\ge0}$ if there is no confusion occurred.
For $\mu\in\mathscr P(\R^d)$, we set $\mu P^{(\delta)}_{n\delta}:=\mathscr L_{X_{n\delta}^{\delta,\mu}}$ to simplify  the notation. $\mathbb W_2$ stands for the $L^2$-Wasserstein distance on $\mathscr P_2(\R^d)$, which is defined as below:
\begin{align*}
\mathbb W_2(\mu,\nu):=\inf_{\pi\in\mathscr C(\mu,\nu)}\bigg(\int_{\R^d\times\R^d}|x-y|^2\pi(\d x,\d y)\bigg)^{\frac{1}{2}},\quad \mu,\nu\in\mathscr P_2(\R^d),
\end{align*}
where $\mathscr C(\mu,\nu)$ means the collection of couplings   between $\mu$ and $\nu.$

Once we choose   the modified EM scheme \eqref{E2} to  state the main result  concerned with the  weak $L^2$-Wasserstein contraction,
there are  plenty of  cumbersome and inexplicit conditions imposed on $b^{(\delta)}$ (rather than $b$); see Theorem \ref{thm}  in Section \ref{sec2}
for more details. So, for the sake of clarity and readability,    we prefer to  opt for some competitive  numerical schemes to present the main result in an explicit and  succinct way.
As for the EM scheme \eqref{WW},   the weak $L^2$-Wasserstein contraction has been  investigated  in depth in \cite{LMM},
where   the drift $b$   is Lipschitz continuous. In this work, we are still interested in the same issue
 but focus on the TEM  scheme \eqref{EW0-} and the PEM scheme  \eqref{EWW}, in which
the drift term $b$ under consideration  is allowed to be   super-linear  and  non-convex in some scenarios (e.g., $b(x)=-\nn U(x)$ for a smooth potential $U:\R^d\to\R$).  To this end, some explicit technical conditions on $b$ are necessary to
 be imposed.

Concerning the iteration \eqref{EW0-}, we assume that
\begin{enumerate}
\item[$({\bf A}_{2 })$]there exists a constant $L_{1} >0$
such that for all $x,y\in\R^d,$
\begin{align}\label{EE10-}
\big|b(x)|y|^{\ell_0}-b(y)|x|^{\ell_0}\big|\le L_{1}\big(1+ |x|^{\ell_0} + |y|^{\ell_0}+ |x|^{\ell_0} |y|^{\ell_0}\big)|x-y|;
\end{align}

\item[$({\bf A}_{3 })$]there exist constants $L_{2},L_{3},L_{4}>0$ and $R^*\ge0$ such that  for all $x,y\in\R^d$ with $|x|>R^*$  or $|y|>R^*$,
\begin{align}\label{EE11-}
\<x-y,b(x)-b(y)\>\le -L_{2}\big(1+|x|^{\ell_0}+|y|^{\ell_0}\big)|x-y|^2,
\end{align}
and
\begin{align}\label{EE12-}
\big\<x-y, b(x)|y|^{\ell_0}-b(y)|x|^{\ell_0}\big\>\le \big(L_{3}\big(1+|x|^{\ell_0}+|y|^{\ell_0}\big)-L_{4} |x|^{\ell_0}|y|^{\ell_0}\big) |x-y|^2.
\end{align}
\end{enumerate}

\ \

As far as the PEM scheme \eqref{EWW} is concerned, we suppose that
 \begin{enumerate}
\item[$({\bf A}_4)$] there exist constants $R_*, L_5=L_5(R_*)>0$ such that for all $x,y\in\R^d$ with $|x|>R_* $  or $|y|>R_*$,
\begin{align}\label{B4}
\<x-y,b(x)-b(y)\>\le  - L_5 |x-y|^2.
\end{align}

\end{enumerate}

The first main result in this work  is described  as follows.

\begin{theorem}\label{thm-0}
	Assume  $(i)$ $({\bf A}_0)$, $({\bf A}_{2})$ and $({\bf A}_{3})$ for the TEM scheme \eqref{EW0-};  $(ii)$ $({\bf A}_0)$ and $({\bf A}_4)$ for the PEM scheme \eqref{EWW}. Then, there exist constants $C_0,\si_0,\ll,\dd_1^\star>0$  such that for $\delta\in(0,\delta_1^\star]$,  $n\ge0$,   $\mu,\nu\in\mathscr P_2(\R^d)$, and the noise intensity $\si$   satisfying
 $|\si|\ge \si_0$,
\begin{align}\label{RT*1}
\mathbb W_2\big(\mu P^{(\delta)}_{n\delta},\nu P^{(\delta)}_{n\delta}\big)\le C_0\e^{-\lambda n\delta}\mathbb W_2(\mu,\nu).
\end{align}
\end{theorem}

Hereinafter, we make some interpretations  on Theorem \ref{thm-0} and the associated assumptions.

\begin{remark}

To state Theorem \ref{thm-0} in an elegant manner,   we therein   don't give the explicit expressions of the quantities $C_0,\lambda, \si_0,\delta_1^\star$ since they  are a little bit lengthy.  However, for completeness, we herein provide their associated concrete forms.
For the TEM scheme \eqref{EW0-}, $\lambda,C_0$ and $\si_0$ are written explicitly  as below:
\begin{equation}\label{DD}
\begin{split}
\lambda  &=\frac{1}{2}\bigg(   (K_R^*/4 )\wedge \Big(6C_R\si^2\Big(\si^2+\frac{24C_R(K_R^*+96 C_R)(1+R)^2}{dK_R^*}\Big)^{-1}-3C_R\Big)\bigg),\\
C_0&=\bigg(1+\frac{24C_R(K_R^*+96 C_R)(1+R)^2}{dK_R^*\si^2}\bigg)^{\ff12},\\
\si_0&=\bigg(\frac{8(1+ R)}{dK_R^*}\big( (3C_R(K_R^*+96 C_R)(1+R))   \vee     (\psi_1( R)K_R^*)  \vee(12C_R  \psi_1(R_0) )  \big)\bigg)^{\ff12} ,
\end{split}
\end{equation}
where
$
R:=R^*, C_R:= (L_0+ L_1 )(1+R^{\ell_0})^2,    K_R^*:=(L_{2 }/2)\wedge L_{4 },
$
and $R_0 : =2\big((1+  R)(1+ 96 C_R/{K_R^*})+  \psi_1(1)\big)$ with $\psi_1(r):=L_0 (r^{\ell_0}+1)r+|b({\bf 0})|,r\ge0$.
On the other hand, with regarding to the PEM scheme \eqref{EWW},  the parameters  $\lambda, C_0$  and $\si_0 $ can also be given concretely via \eqref{DD} with  the constants $R,C_R, K_R^*$  being replaced respectively by the following ones:
$
R=R_*,  C_R=2L_0\varphi(R)$ and $    K_R^*=L_5.
$
In addition, for the TEM scheme \eqref{EW0-} and the PEM scheme \eqref{EWW},
 the quantity  $\delta_1^\star$ is also  given explicitly  in \eqref{dd1} with $ \theta=\gamma, \psi_2(r) \equiv L_0$ and respective   $   \delta_R=(L_{2}/(2L_{3}))^{\frac{1}{\gamma}}\wedge1,\delta_r^*\equiv1$ and $  \delta_R=\varphi(R)^{-\frac{1}{\gamma}},\delta_r^*=\varphi(r)^{-\frac{1}{\gamma}}.$

At first sight,  Assumptions $({\bf A}_{2})$ and $({\bf A}_{3})$ seem to be somewhat remarkable. Nevertheless, due to the structure of the TEM scheme \eqref{EW0-},
 they are imposed  quite naturally
when we examine Assumption (${\bf H}_1$) in Section \ref{sec2} for $\pi^{(\delta)}={\rm Id}$ and $b^{(\delta)}$ defined in \eqref{EY-}. $({\bf A}_{4})$ illustrates that $b$ is dissipative merely   outside of a closed ball. Additionally, $({\bf A}_{4})$, along with  $({\bf A}_{0})$, implies the validity of (${\bf H}_1$) with $b^{(\delta)}=b$ and $\pi^{(\delta)}$ given in \eqref{ET-}.
\end{remark}

Below, we shall present some potential applications of Theorem \ref{thm-0}. Since $\delta_xP_\delta^{(\delta)}$ is Gaussian,
it satisfies the Poincar\'{e} inequality (see, for instance, \cite[Proposition 4.1.1]{BGL}).  The aforementioned fact, together with the weak $L^2$-Wasserstein contraction \eqref{RT*1} (which implies the $L^2$-gradient estimate), yields that $(\delta_xP_{n\delta}^{(\delta)})_{n\ge1}$ and the IPM of $(X_{n\delta}^\delta)_{n\ge1}$
satisfy respectively  the Poincar\'{e} inequality; see, for example, \cite[Theorem 2.6]{LMM}. Additionally, $\delta_xP_\delta^{(\delta)}$, which is Gaussian as stated previously, satisfies a log-Sobolev inequality  \cite[Proposition 5.5.1]{BGL}. In turn, the Talagrand inequality (which is also called the transportation cost-information inequality in literature) is valid. Subsequently, the weak $L^2$-Wasserstein contraction (so the weak $L^1$-Wasserstein contraction is available) implies Gaussian concentration inequalities for empirical averages  \cite[Proposition 7.1]{LMM}.
For further applications concerning non-asymptotic convergence  bounds in  the KL-divergence, we would like to refer to \cite[Section $8$]{LMM} for more details.

In this paper, we shall  give an account of two additional applications of Theorem \ref{thm-0}, which include the non-asymptotic $L^2$-Wasserstein bounds between the exact IPMs and the numerical transition kernels, and the strong law of large numbers (LLN for short) of the additive functionals corresponding to the EM scheme, the TEM scheme as well as the PEM scheme, respectively.

\subsection{Application I: non-asymptotic $L^2$-Wasserstein bound}
Concerning the TEM scheme   associated with the Langevin SDE,
 \cite[Theorem 5]{BDMS} revealed that the convergence rate (with respect to the step size) of the non-asymptotic $L^2$-Wasserstein  bound
is $1/2$ under the {\it uniformly dissipative condition}. After that,  the uniform dissipation in \cite{BDMS} was weaken to the partially dissipative setting in \cite{DAS}. Nevertheless,   the corresponding convergence rate
 is $1/4$ (see \cite[Corollary 2.10]{DAS}), where an ingredient relies on that  the $L^2$-Wasserstein distance can be controlled by a quasi-Wasserstein distance
  (see \cite[Lemma A.3]{DAS}).

  In this subsection, our goal is to further improve e.g. \cite{BDMS,DAS}.
Besides $({\bf A}_0)$, $({\bf A}_2)$, $({\bf A}_3)$ as well as $({\bf A}_4)$, we additionally   assume that
\begin{enumerate}
 \item[$({\bf A}_5)$]
 there exist constants $  L_{6}, \ell_0^\star\ge 0$ such that for all $x,y\in \R^d$
\begin{align*}
\|\nn b(x)-\nn b(y)\|_{\rm op}	\le L_{6}\big(1+|x|^{\ell_0^\star}+|y|^{\ell_0^\star}\big)|x-y|,
\end{align*}
  where $\|\cdot\|_{\rm op}$ denotes the operator norm  and $\nn $ stands for the weak gradient operator.
\end{enumerate}

\begin{theorem}\label{cor}
Assume that $(i)$ $({\bf A}_0)$ with $\ell_0=0$, $({\bf A}_{4})$ and $({\bf A}_{5})$ with $\ell_0^\star=0$ for the EM scheme \eqref{WW};   $(ii)$ $({\bf A}_0)$, $({\bf A}_2)$ and $({\bf A}_{3})$ for the TEM scheme \eqref{EW0-};   $(iii)$ $({\bf A}_0)$, $({\bf A}_{4})$ and $({\bf A}_{5})$ for the PEM scheme \eqref{EWW}. Then, there exist constants $C_0,\si_0,\ll>0, \delta_2^{\star}\in(0,1]$ such that for all  $\dd\in(0,\dd_2^{\star}]$,   $\mu\in \mathscr P_2(\R^d)$, and the noise intensity $\si$ satisfying $|\si|\ge\si_0,$
\begin{align}\label{RT2}
\W_{2}\big(\pi_\8,\pi_\8^{(\delta)}\big)\le  C_0 \delta^{ \gamma_\star}  d^{\ell_\star},
\end{align}
and
\begin{align}
\W_2(\mu P^{(\dd)}_{n\dd},\pi_\8)\le C_0 \big(\e^{-\ll n\dd}	\W_2(\mu,\pi_\8)+\dd^{\gamma_\star}d^{\ell_\star}\big),
\end{align}
where $\pi_\8$ $($resp. $\pi_\8^{(\delta)}$$)$ is the unique IPM  of $(X_t)_{t\ge0}$ $($resp. $(X_{n\delta}^{\delta})_{n\ge0}$$)$, and
 $\gamma_\star=1,\ell_\star=1/2$  for the EM scheme \eqref{WW}; $\gamma_\star=\gamma, \ell_\star=\ell_0+1/2$ for the TEM scheme \eqref{EW0-}; $\gamma_\star=1,\ell_\star=(\ell_0+1+{\ell_0^\star}/2)\vee((\ell_0+1)/2+\ell_0/\gamma)$ for the PEM scheme \eqref{EWW}.
\end{theorem}

Below, we offer some further explanations for Theorem \ref{cor}.
\begin{remark}
In comparison to \cite{DAS},  with regard to the same TEM scheme \eqref{EW0-}, in Theorem \ref{cor}
we have improved the associated convergence rate to $\gamma$, which is  close to $1/2$.
In Theorem \ref{cor}, the positive constant $C_0$ is  dimension-free,   and the number $\delta_2^{\star }=\delta_1^\star\wedge\big( K_R^*/(16L_0^2)\big)^{\frac{1}{1-2 \theta}}$, where  $\delta_1^\star\in(0,1]$ is the same as the one in Theorem \ref{thm-0} and the parameter $\theta\in[0,1/2)$ can be traced in Lemma \ref{lem1}.
On the other hand,
the parameters $\lambda,\si_0>0$   are  the same as those given in Theorem \ref{thm-0} for the TEM scheme \eqref{EW0-} and the PEM scheme \eqref{EWW}, respectively. In addition, with regarding to  the EM scheme \eqref{WW}, the quantities $\lambda,\si_0$
are  defined as in \eqref{DD} for   $R=R_*$, $C_R=K_R=3L_0$ and $K_R^*=L_5$. Finally, we would like  to stress that Assumption $({\bf A}_5)$ is enforced just to improve the corresponding convergence rate. That is to say, Assumption $({\bf A}_5)$  can be dropped definitely once one doesn't care about the higher order of  the associated convergence rate.
\end{remark}

In \cite{NMZ}, a new variant of the classical TEM scheme was proposed to improve the corresponding convergence rate of the non-asymptotic $L^2$-Wasserstein bound.
 Inspired by  \cite{NMZ}, we introduce    the following TEM algorithm: for any $n\ge0 $ and $\delta>0,$
\begin{align}\label{RT3}
X_{(n+1)\delta}^{\delta}= X_{n\delta}^{\delta} +\bar b^{(\delta)}( X_{n\delta}^{\delta} )\delta+\si\triangle W_{n\delta},
\end{align}
in which
\begin{align}\label{TY}
 \bar b^{(\delta)}(x):=
\frac{b(x)}{(1+\delta|x|^{2\ell_0})^{\frac{1}{2}}},\quad x\in\R^d.
\end{align}

Concerning  \eqref{RT3}, it is a formidable task to examine Assumption (${\bf H}_1$)  in Section \ref{sec2}. In turn, we have recourse to  the weak $L^2$-Wasserstein contraction of $(X_t)_{t\ge0}$  to explore directly the associated non-asymptotic $L^2$-Wasserstein bound.
In comparison with \cite[Theorem 2.10]{NMZ}, where the underlying convergence rate is $\frac{1}{2}$,
the theorem below (which is also very interesting in its own right)
demonstrates that the corresponding $L^2$-Wasserstein  convergence rate is $1$.
 \begin{theorem}\label{cor1-1}
Assume that $({\bf A}_0)$, $ ({\bf A}_2)$,  $ ({\bf A}_3)$ and $ ({\bf A}_5)$ hold.
Then,  there are constants $C_0,\si_0,\lambda_\star>0$   such that  for all
$\dd\in(0,\dd_{3}^\star]$, $\mu\in \mathscr P_{p_\star}(\R^d)$ and the noise intensity $\si$ satisfying  $|\si|\ge\si_0$,
 \begin{align}\label{EW**}
	\W_2(\mu P_{n\dd}^{(\dd)},\pi_\8)\le  C_0\big(\e^{-\ll_\star n\dd}\W_2(\mu,\pi_\8)+\dd d^{\ell_{\star\star}}\big),
\end{align}
where $\pi_\8$ is the unique IPM of $(X_t)_{t\ge0}$, and
\begin{equation}\label{DD6}
\begin{split}
\delta^\star_3:&= (L^2_2/(32L_0^4))\wedge(2\ss2/{L_2})\wedge (1/2),\quad   p_\star:=2(4\vee(2\ell_0^\star)\vee( 3\ell_0+1 ))
\\
 \ell_{\star\star}:&= (1+3\ell_0)\vee (\ell_0^\star+2).
\end{split}
\end{equation}
Consequently,
\begin{align} \label{WE}
	\W_2(\pi_\8^{(\delta)},\pi_\8)\le  C_0\dd  d^{\ell_{\star\star}},
\end{align}
where $\pi_\8^{(\delta)}$ is the IPM of $(X_{n\delta}^\delta)_{n\ge0}$ determined by \eqref{RT3}.
\end{theorem}

\begin{remark}
The constants $\lambda_\star,\si_0$ have been given explicitly in \cite[Theorem 1]{PM}. Once more, we stress that the positive constant $C_0$ in Theorem \ref{cor1-1} is dimension-free. In contrast to the TEM scheme \eqref{EW0-}, the non-asymptotic $L^2$-Wasserstein bound concerned with the newly designed algorithm \eqref{RT3}   enjoys a faster convergence rate. Unfortunately, the latter one requires the higher moment with regard to initial distributions. Whereas, as far as the schemes \eqref{WW}, \eqref{EW0-} and \eqref{EWW} are concerned, the requirement on finite second-order  moment of initial distributions is enough to investigate the associated non-asymptotic $L^2$-Wasserstein bound as demonstrated in Theorem \ref{cor}.
\end{remark}

Before the end of this subsection,
we make the following  table and  compare clearly  the results derived in the present work
with the existing literature based on   various aspects (e.g., the technical assumption, the convergence rate and the dimension dependency).

\smallskip

\renewcommand{\arraystretch}{1.2}
	\begin{center}\label{tab}
\begin{tabular}{llllc}
        \toprule
        \textbf{Source} & \textbf{Algorithm} & \textbf{Convexity} & \textbf{order}& \textbf{Dependence on $d$}\\
        \midrule
        \cite[Theorem 6]{BDMS} & TEM & strong convex &$\dd$&-\\
        \cite[Corollary 2.10]{DAS} & TEM& non-convex & $\dd^{\frac 14}$&$\e^{cd}$\\
        \cite[Theorem 2.10]{NMZ} & TEM &non-convex &$\dd^{\frac12}$&$\e^{cd}$\\
        \cite[Corollary 9]{DM} & EM & strong convex & $\dd$&$d$ \\
        \midrule
         \textbf{Theorem \ref{cor}} & EM & non-convex&$\delta$&$ d^{\ff12}$\\
        \textbf{Theorem \ref{cor}} & TEM \eqref{EW0-}& non-convex&$\delta^\gamma $&$ d^{\ell_0+1/2}$\\
         \textbf{Theorem \ref{cor1-1}} & TEM \eqref{RT3}& non-convex& $\delta $&$d^{(3\ell_0+1)\vee ({\ell_0^\star}+2)} $\\
          \textbf{Theorem \ref{cor}} & PEM& non-convex &$\delta $&$d^{(\ell_0+1+{\ell_0^\star}/2)\vee((\ell_0+1)/2+\ell_0/\gamma)}$\\

        \bottomrule
    \end{tabular}

     \parbox{\linewidth}{\centering \textbf{Table}:  The dimension dependency and the convergence rate of $\mathbb W_2(\mathscr L_{X_{n\dd}^\dd},\pi_\8)$.}
\end{center}

\subsection{Application II: strong law of large numbers}
So far, the issue on limit theorems for continuous-time Markov processes has been studied very intensively; see, for instance, \cite{JS,KW,KLO,Kulik,Sh}. Recently, in \cite{BH} we investigated the  central limit theorems and the strong LLN for SDEs with irregular drifts, which allow the drifts involved to be H\"older continuous or piecewise continuous. In the meantime, there is a huge literature on the establishment of limit theorems for stochastic algorithms; see \cite{CDH,JD,LX,PR}, to name just a few. In particular, concerning the numerical version of the  LLN, most of the literature focuses on the case that the drifts under investigation  are {\it globally dissipative}; see, for example, \cite{CDH,JD}. Indeed, under the globally dissipative condition,  the classical synchronous coupling approach has been employed to handle the  ergodicity  of continuous-time stochastic systems under consideration. Unfortunately, such an approach no longer works once the underlying stochastic systems   are partially dissipative.

As a further application of Theorem \ref{cor} (so Theorem \ref{thm-0}),  in this subsection  we make an attempt to establish the strong  LLN   of the additive functionals associated with the   schemes \eqref{WW}, \eqref{EW0-} and \eqref{EWW}, where the drifts   might be {\it dissipative merely outside of a ball}. For this purpose, we introduce the following function class.  For  $\rho\ge 0$, we define the set $\mathscr C_{\rho}$ by
\begin{align}\label{f1}
\mathscr C_\rho=\bigg\{f:\R^d\to\R\Big|\|f\|_\rho:=\sup_{x\neq y}\ff{\big|f(x)-f(y)\big|}{|x-y|(1+|x|^\rho+|y|^\rho)}<\8\bigg\}.
\end{align}

The following theorem demonstrates that the associated  additive functional (i.e., the time average) converges a.s.  to the spatial average with respect to the equilibrium, and  provides  the corresponding  convergence rate.
\begin{theorem}\label{LLN0}
Assume that $(i)$ $({\bf A}_0)$ with $\ell_0=0$, $({\bf A}_{4})$ and $({\bf A}_{5})$ with $\ell_0^\star=0$ for the EM scheme \eqref{WW};   $(ii)$ $({\bf A}_0)$, $({\bf A}_2)$ and $({\bf A}_{3})$ for the TEM scheme \eqref{EW0-};   $(iii)$ $({\bf A}_0)$, $({\bf A}_{4})$ and $({\bf A}_{5})$ for the PEM scheme \eqref{EWW}.
Then, for any $f\in\mathscr C_\rho$, $\vv\in(0,1/2)$, and $x\in\R^d,$ there exist constants $C_0=C_0(x,\rho,\vv,\|f\|_\rho),\si_0>0$ and  a random time $N_{\vv,\dd}=N_{\vv,\dd}(x,\rho,d)\ge 1$    such that for all $\dd\in(0,\dd_2^\star]$,  $n\ge N_{\vv,\dd}$ and the noise intensity  $\si$ satisfying  $|\si|\ge\si_0,$
\begin{equation}\label{EE1*}
\bigg|\frac{1}{n}\sum_{k=0}^{n-1} f(X_{k\dd}^{\dd,x})-\pi_\8(f) \bigg|\le C_0 \big(n^{-{1}/{2}+\vv}\dd^{-1/2}+\dd^{\gamma_\star}   d^{(\ell_\star+ \rho)/2}\big),\quad a.s.,
\end{equation}
where
$\pi_\8$ means the unique IPM of $(X_{t})_{t\ge 0}$ solving \eqref{E1}, and the  quantities $\gamma_\star,\ell_\star$ are the same as those given in Theorem \ref{cor}. Moreover, for any $q>0,$ there is a constant $C_q^*>0$ such that
\begin{align}\label{7W*}
	\E N_{\vv,\dd}^q\le C^*_q\big(1+|x|^{2(1+\rho)} +d^{1+\rho}\big)^{\frac{q+2}{2\vv}}.
\end{align}

\end{theorem}

\begin{remark}
For the modified EM scheme \eqref{E2},   we build   a general framework to establish the strong LLN; see Section \ref{sec5} for more details.
Satisfactorily, Theorem \ref{LLN0}  demonstrates that the convergence rate with respect to the number of iteration $n$ is nearly optimal.
Additionally,  Theorem \ref{LLN0} is  applicable to more examples which are excluded in e.g. \cite{CDH,JD}.
\end{remark}

The content of this paper is organized as follows.
Section \ref{sec2} is devoted to providing criteria  on the weak $L^2$-Wasserstein contraction of the modified EM scheme \eqref{E2}. Based on this,
the proof of Theorem \ref{thm-0} is finished in Section \ref{sec2}. In Section \ref{sec3}, we furnish sufficient conditions to derive  the non-asymptotic $L^2$-Wasserstein bounds of the modified EM algorithm \eqref{E2}. Subsequently, the proof of Theorem \ref{cor} is done. In Section \ref{sec4}, we aim at carrying out the proof of Theorem \ref{cor1-1} with the aid of the $L^2$-Wasserstein contraction of $(X_t)_{t\ge0}$.  In the last section, we  establish the  strong LLN for the scheme  \eqref{E2}. As an application,   the proof of Theorem \ref{LLN0} is complete.

\section{  Criteria on weak $L^2$-Wasserstein contraction and Proof of Theorem \ref{thm-0} }\label{sec2}

At first glance, the modified EM scheme \eqref{E2} is unusual. So, before we move on,
it is primary to address the issue that, under what suitable conditions, the
  scheme \eqref{E2} is non-explosive in a finite horizon. To this end, we impose  the following assumption:
\begin{enumerate}
\item[(${\bf H}_0$)]there exist constants $c_*,c^*>0$ and $\kk\in(0,1/2]$ such that for all $\delta\in(0,1]$ and $x\in\R^d,$
\begin{align*}
|b^{(\delta)}(x)|\le c_*+c^*\delta^{-\kk}|x|\quad \mbox{ and } \quad \<x,b^{(\delta)}(x)\>\le c_*+c^*|x|^2.
\end{align*}
\end{enumerate}

The following lemma  demonstrates that the scheme \eqref{E2} has a finite second-order moment (so it is not explosive) in a finite-time interval.

\begin{lemma}\label{nonexplosition}
Under $({\bf H}_0)$  and $({\bf A}_1)$, it holds that for any $\delta\in(0,1]$, $x\in\R^d,$ and $n\ge0,$
\begin{align}\label{ET*}
\E|X_{n\delta}^{\delta,x}|^2\le\bigg(|x|^2+\frac{2c_*(1+c_*)+\si^2d}{2c^*(1+ c^*)}\bigg)\e^{2c^*(1+ c^*)n \delta}.
\end{align}

\end{lemma}

\begin{proof}
In terms of  the expression of the modified EM scheme given in \eqref{E2}, it is easy to see that
\begin{align*}
\E\big|X_{(n+1)\delta}^{\delta,x}\big|^2&=\E\big|\pi^{(\delta)}(X_{n\delta}^{\delta,x})\big|^2
+2\E\<\pi^{(\delta)}(X_{n\delta}^{\delta,x}),b^{(\delta)}(\pi^{(\delta)}(X_{n\delta}^{\delta,x}))\>\delta\\
&\quad+
\E\big|b^{(\delta)}(\pi^{(\delta)}(X_{n\delta}^{\delta,x}))\big|^2\delta^2 +
\si^2d\delta,
\end{align*}
where we also used the following facts:
\begin{align*}
\E\<\pi^{(\delta)}(X_{n\delta}^{\delta,x})+b^{(\delta)}(\pi^{(\delta)}(X_{n\delta}^{\delta,x}))\>\delta,\triangle W_{n\delta}\>=0\quad \mbox{ and } \quad \E|\triangle W_{n\delta}|^2=d\delta.
\end{align*}
Next, by virtue of $({\bf H}_0)$, along with $\kk\in(0,1/2]$, $|\pi^{\delta}(x)|\le|x|$ as well as the  inequality: $(a+b)^2\le 2(a^2+b^2), a,b\in\R$, it follows that for all   $\delta\in(0,1],$
\begin{align*}
\E\big|X_{(n+1)\delta}^{\delta,x}\big|^2
&\le\big(1+2c^*\delta+2(c^*)^2\delta^{2(1-\kk)}\big)\E\big| \pi^{(\delta)}(X_{n\delta}^{\delta,x})\big|^2+(2c_*(1+c_*)+\si^2d)\delta\\
&\le \big(1+2c^* (1+  c^*) \delta\big)\E\big| X_{n\delta}^{\delta,x}\big |^2+(2c_*(1+ c_*)+\si^2d)\delta.
\end{align*}
 Thereafter, an inductive argument, besides the basic inequality: $1+r\le \e^{r}, r\ge0,$ shows that
\begin{align*}
\E\big|X_{(n+1)\delta}^{\delta,x}\big|^2&\le \big(1+2c^*(1+ c^*) \delta\big)^{n+1}|x|^2+(2c_*(1+ c_*)+\si^2d)\delta\sum_{i=0}^n\big(1+2c^*(1+ c^*) \delta\big)^i\\
&\le \bigg(|x|^2+\frac{2c_*(1+ c_*)+\si^2d}{2c^*(1+ c^*)}\bigg)\big(1+2c^*(1+ c^*) \delta\big)^{n+1}\\
&\le \bigg(|x|^2+\frac{2c_*(1+ c_*)+\si^2d}{2c^*(1+ c^*)}\bigg)\e^{2c^*(1+ c^*)(n+1) \delta}.
\end{align*}
This thus  yields the desired assertion \eqref{ET*} right now.
\end{proof}

\begin{remark}\label{rem2}
It is ready to see that, as regards the TEM scheme \eqref{EW0-} and the PEM scheme \eqref{EWW},
(${\bf H}_0$) is fulfilled as soon as $({\bf A}_0)$  and  
the drift condition
$\<x,b(x)\>\le c_1+c_2|x|^2,x\in\R^d,$ holds true for some positive constants $c_1,c_2.$ Apparently, (${\bf H}_0$)
implies the following one:
\begin{enumerate}
\item[(${\bf H}_0'$)]there exist constants $c_*,c^*>0$ and $\kk\in(0,1/2]$ such that for all $\delta\in(0,1]$ and $x\in\R^d,$
\begin{align*}
|b^{(\delta)}(\pi^{(\delta)}(x))|\le c_*+c^*\delta^{-\kk}|\pi^{(\delta)}(x)|\quad \mbox{ and } \quad \<\pi^{(\delta)}(x),b^{(\delta)}(\pi^{(\delta)}(x))\>\le c_*+c^*|\pi^{(\delta)}(x)|^2.
\end{align*}
\end{enumerate}
As   a matter of fact, by checking the course of the proof of Lemma \ref{nonexplosition},
 Assumption (${\bf H}_0'$)  is sufficient for our purpose. By invoking \eqref{E0}, \eqref{E*} and   \eqref{EY*}
below, Assumption (${\bf H}_0')$ can be guaranteed; see Remark \ref{re1} (1) for more details. Hence, the main result (i.e., Theorem \ref{thm}) presented in this section can make sense due to the non-explosion of the scheme \eqref{E2}.
\end{remark}

To avoid proving  Theorem \ref{thm-0}    on a case-by-case basis, in this section we intend to  establish  a corresponding  general
result (see Theorem \ref{thm} below) for the modified EM scheme \eqref{E2}.  As long as the general result is available, the proof  of
 Theorem \ref{thm-0}  can be finished  via checking successively the associated   conditions.
To explore the $L^2$-Wasserstein contraction of $(X_{n\delta}^\delta)_{n\ge0,}$ we enforce some technical  conditions on the modified drift $b^{(\delta)}$
and the noise intensity $\si.$ More precisely, we suppose    that
\begin{enumerate}
\item[(${\bf H}_1$)] there exist constants $R, C_R,K_R>0$,   $ \theta\in[0,1/2)$, and $\delta_R\in(0,1]$ such that for any $\delta\in(0,\delta_R]$, and  $x,y\in\R^d,$
\begin{equation}\label{E0}
\big|b^{(\delta)}(\pi^{(\delta)}(x))-b^{(\delta)}(\pi^{(\delta)}(y))\big|\le
\begin{cases}
 C_R|x-y|,\quad \quad~~ |x|\le R, |y|\le R,\\
  K_R\delta^{- \theta} |x-y|, \quad |x|>R \mbox{ or } |y|>R;
\end{cases}
\end{equation}
moreover, there exists a  constant   $ K_R^* >0$   such that for all $\delta\in(0,\delta_R]$, and $x,y\in\R^d $  with $|x|>R$  or $|y|>R$,
\begin{align}\label{E*}
\big\< \pi^{(\delta)}(x)-\pi^{(\delta)}(y), b^{(\delta)}(\pi^{(\delta)}(x))-b^{(\delta)}(\pi^{(\delta)}(y))\big\>\le -K_R^* \big| \pi^{(\delta)}(x)-\pi^{(\delta)}(y)\big|^2;
\end{align}
\item[(${\bf H}_2$)]
for any $r> 0,$ there exists a constant $  \delta_{  r}^* \in(0,1]$ such that for any $\delta\in(0,  \delta_{r}^* ]$,
\begin{align}\label{EY}
 \inf_{|x|\ge r}\big|\pi^{(\delta)}(x)\big|\ge r\quad \mbox{ and } \quad  |\pi^{(\delta)}(x)| = |x|  \quad \mbox{ for } |x|\le r;
\end{align}
in addition, there exist increasing positive functions $[ 0,\8)\ni r \mapsto \psi_1(r),  \psi_2(r)  $
  such that for any $\delta\in(0,1]$, $x\in\R^d,$ and $r\ge 0,$
\begin{align}\label{EY*}
\big|b^{(\delta)}(\pi^{(\delta)}(x))\big|\le   \psi_1(r)  +   \psi_2(r)  \delta^{-\theta}|\pi^{(\delta)}(x)|\I_{\{|\pi^{(\delta)}(x)|>r\}}.
\end{align}

\item[(${\bf H}_3$)] the noise intensity $\si$ satisfies
\begin{align*}
\si^2>\frac{8(1+ R)}{dK_R^*}\big( (3C_R(K_R^*+96 C_R)(1+R))   \vee     (\psi_1( R)K_R^*)  \vee(12C_R  \psi_1(R_0) )  \big),
\end{align*}
in which 
$
R_0 : =2\big((1+  R)(1+ 96 C_R/{K_R^*})+  \psi_1(1)\big).
$

\end{enumerate}

Below,  we make some comments regarding Assumptions $({\bf H}_1 )$-$({\bf H}_3)$.
\begin{remark}\label{re1}

\begin{enumerate}
\item[(1)] By virtue of \eqref{EY*} with $r=1$, it is easy to see that for  $\delta\in(0,1]$ and $x\in\R^d,$
\begin{align*}
\big|b^{(\delta)}(\pi^{(\delta)}(x))\big|\le   \psi_1(1)  +   \psi_2(1)  \delta^{-\theta}|\pi^{(\delta)}(x)|.
\end{align*}
Next, we obtain from \eqref{E0} and \eqref{E*} that  for all  $\delta\in(0,1]$ and $x\in\R^d,$
\begin{equation*}
\begin{split}
\<\pi^{(\delta)}(x),b^{(\delta)}(\pi^{(\delta)}(x))\>&=\<\pi^{(\delta)}(x),b^{(\delta)}(\pi^{(\delta)}(x))-b^{(\delta)}({\bf0})\>+\<\pi^{(\delta)}(x), b^{(\delta)}({\bf0})\>\\
&\le |\pi^{(\delta)}(x)|\cdot|b^{(\delta)}(\pi^{(\delta)}(x))-b^{(\delta)}({\bf0})|\I_{\{|x|\le R\}}\\
&\quad-K_R^*|\pi^{(\delta)}(x)|^2\I_{\{|x|> R\}}+|\pi^{(\delta)}(x)|\cdot|b^{(\delta)}({\bf0})|\\
&\le (C_R+K_R^*)R^2+(1/2-K_R^*)|\pi^{(\delta)}(x)|^2+ |b^{(\delta)}({\bf0})|^2/2,
\end{split}
\end{equation*}
where in the last inequality we used $|\pi^{(\delta)}(x)|\le |x|$.
Then, due to  $\sup_{\delta\in(0,1]}|b^{(\delta)}({\bf0})|<\8,$
 Assumption $({\bf H}_0')$ is satisfied so    the modified EM scheme \eqref{E2} is non-explosive in a finite-time interval by taking advantage of Lemma \ref{nonexplosition} and Remark \ref{rem2}.

\item[(2)]
 \eqref{E0} shows that $b^{(\delta)}\circ\pi^{(\delta)}$ is globally Lipschitz, where the underlying Lipschitz constant is dependent on the step size when one of the spatial variables is located outside of a compact set. Moreover,
   \eqref{E*} reveals that $b^{(\delta)} $ is dissipative merely in the outside of a compact set. Assumption   $({\bf H}_1)$ seems to be  unconventional. Whereas, $({\bf H}_1)$ can be satisfied by several well-known numerical approximation schemes. For instance,
   in case of  $b^{(\delta)} =b  $ and $\pi^{(\delta)} ={\rm Id}$ (which correspond to the classical EM scheme),
 Assumption (${\bf H}_1$)  with $\theta=0$  reduces to $({\bf A}_2)$ and \eqref{EE11-} with $\ell_0=0$, which
  coincide exactly with \cite[Assumption 1 and Assumption 2]{LMM}. Furthermore,
  concerning (i) $\pi^{(\delta)}={\rm Id}$ and $b^{(\delta)}$ defined in \eqref{EY-}, and  (ii) $b^{(\delta)} =b $
and $\pi^{(\delta)} $ given in \eqref{ET-} (which correspond respectively   to the TEM scheme and the PEM scheme), Assumption $({\bf H}_1)$ is also fulfilled under (i) (${\bf A}_0$), (${\bf A}_2$) and  (${\bf A}_3$), as well as  (ii) (${\bf A}_0$) and   (${\bf A}_4$), separately;
 see Lemma   \ref{lem1}  for more details. Additionally, in \cite[({\bf H1})]{DEEGM},
  a counterpart    of  (${\bf H}_1$)  given as below: for some constants $c_1,c_2,\ell_0>0,$
\begin{align}\label{EP*}
|T_\delta(x)-T_\delta(y)|\le \big(1+c_1\delta \I_{\{|x-y|\le \ell_0\}}-c_2\delta\I_{\{|x-y|> \ell_0\}}\big)|x-y|,~x,y\in\R^d
\end{align}
with  $T_\delta:\R^d\to\R^d$ ($T_\delta(x):=x+b(x)\delta$ for the EM scheme) was exerted  to provide bounds in the Wasserstein distance  and the total variation distance  between the distributions of two functional autoregressive processes. Whereas, some important stochastic algorithms (e.g., the PEM scheme) are ruled out by \eqref{EP*}.

\item[(3)] Assumption $({\bf H}_2)$ is still abstract. However, it is  fulfilled  by     the TEM scheme \eqref{EW0-} and the PEM scheme \eqref{EWW}; see Lemma \ref{lem2} below for related  details.  (${\bf H}_3$) indicates that the noise   is non-degenerate and the associated noise intensity should be strong enough.
 \end{enumerate}
\end{remark}

Under $({\bf H}_1 )$,
the weak contraction  of $(X_{n\delta}^\delta)_{n\ge0}$ determined by \eqref{E2} under different probability distances (including  the  $L^1$-Wasserstein distance plus the total variation, the additive Wasserstein distance, as well as the $L^1$-Wasserstein distance)
 was addressed   in \cite{BMW}.
  Nonetheless,  the exploration  on the weak $L^2$-Wasserstein contraction was left open  therein when the drift terms are of super-linear growth in particular.

 Before we proceed, we introduce some quantities. Set
 \begin{equation}\label{WQ*}
\begin{split}
\lambda_1: &=  6C_R\si^2\Big(\si^2+\frac{24C_R(K_R^*+96 C_R)(1+R)^2}{dK_R^*}\Big)^{-1}-3C_R ,\quad \lambda_2:=K_R^*/4\\
C_0: &=\bigg(1+\frac{24C_R(K_R^*+96 C_R)(1+R)^2}{dK_R^*\si^2}\bigg)^{\ff12},
\end{split}
\end{equation}
where   $\lambda_1$ is positive by taking $({\bf H}_3)$ into account. Moreover, let
 \begin{equation}\label{dd1}
 \begin{split}
\delta_1^\star = &\delta_R\wedge (2K_R^*)^{-1}\wedge (K_R^*/K_R^2)^{\frac{1}{1-2\theta}}\wedge (1/C_R) \wedge (1/{\lambda_1})\wedge (1/{\lambda_2})\\
&\wedge \delta_{  R_0 }^* \wedge   (2 \psi_2(1) )^{\frac{1}{\theta-1}}\wedge \big(4(2+d) (1+ R^2 )\si^2\big)^{-1}.
\end{split}
\end{equation}
Herein, we want to stress that, in some scenarios, the underlying upper bounds of the step size are allowed to be bigger than $\delta_1^\star$ given in \eqref{dd1}. Nevertheless, throughout this section we always use the upper bound $\delta_1^\star$ to avoid introducing too many unimportant quantities.

The main result in this section is stated as follows, where the corresponding proof is deferred to the end of this  section.

\begin{theorem}\label{thm}
Assume that   $({\bf A}_1)$ and $({\bf H}_1)$-$({\bf H}_3)$ hold.
Then,
   for all  $\delta\in(0,\delta_1^\star]$,  $n\ge0$, and $\mu,\nu\in\mathscr P_2(\R^d)$,
\begin{align}\label{E3}
\mathbb W_2\big(\mu P^{(\delta)}_{n\delta},\nu P^{(\delta)}_{n\delta}\big)\le C_0\e^{-\frac{1}{2}(\lambda_1\wedge\lambda_2) n\delta}\mathbb W_2(\mu,\nu),
\end{align}
where $\lambda_1,\lambda_2, C_0$ and $\delta^\star_1$ were defined in \eqref{WQ*} and \eqref{dd1}, respectively.
\end{theorem}

In the sequel, we present three lemmas, which, on the one hand,  provide explicit conditions to examine $({\bf A}_1)$, $({\bf H}_1)$ and $({\bf H}_2)$, and, on the other hand, will play a crucial role in the proof of Theorem \ref{thm-0}.

\begin{lemma}\label{lem0}
$({\bf A}_1)$ holds true for $\pi^{(\delta)}={\rm Id}$ and $\pi^{(\delta)}$ defined in \eqref{ET-}, respectively.
\end{lemma}

\begin{proof}
It is trivial to see that $({\bf A}_1)$ is true for $\pi^{(\delta)}={\rm Id}$. For $\pi^{(\delta)}$ defined in \eqref{ET-}, note that for all $x,y\in\R^d,$
\begin{equation*}
\begin{split}
 |x-y|^2-|\pi^{(\delta)}(x)-\pi^{(\delta)}(y) |^2
&\ge |x|^2-\big(|x|\wedge \varphi^{-1}(\delta^{-\gamma})\big)^2 +|y|^2-\big(|y|\wedge \varphi^{-1}(\delta^{-\gamma})\big)^2\\
&\quad-2\big(|x|\cdot|y|-(|x|\wedge \varphi^{-1}(\delta^{-\gamma}))(|y|\wedge \varphi^{-1}(\delta^{-\gamma})\big)\\
&=:\Lambda(x,y,\delta).
\end{split}
\end{equation*}
Whence, it is sufficient to verify $\Lambda(x,y,\delta)\ge0$ in order to show the validity of $({\bf A}_1)$. Obviously, $\Lambda(x,y,\delta)=0$
for $x,y\in\R^d$ with $|x|\vee|y|\le \varphi^{-1}(\delta^{-\gamma})$. Next, for   $x,y\in\R^d$ with $|x|\wedge|y|\ge \varphi^{-1}(\delta^{-\gamma})$,
it follows that $\Lambda(x,y,\delta)=(|x|-|y|)^2\ge0.$ Furthermore, for  $x,y\in\R^d$  with $|x|\le\varphi^{-1}(\delta^{-\gamma})$ and $|y|\ge\varphi^{-1}(\delta^{-\gamma})$, we find that
\begin{align*}
\Lambda(x,y,\delta)=\big(|y| + \varphi^{-1}(\delta^{-\gamma}) -2|x|\big)\big( |y|-  \varphi^{-1}(\delta^{-\gamma})\big)\ge0.
\end{align*}
Thus, by interchanging $x$ and $y$, we can conclude that $\Lambda(x,y,\delta)\ge0$ for $x,y\in\R^d$ with $|x|\ge\varphi^{-1}(\delta^{-\gamma})$
or $|y|\ge\varphi^{-1}(\delta^{-\gamma})$.
\end{proof}

\begin{lemma}\label{lem1}
Under $({\bf A}_0)$, $({\bf A}_2)$ and $({\bf A}_3)$, for  $b^{(\delta)}$ given in \eqref{EY-} and $\pi^{(\delta)}={\mbox Id}$,
$({\bf H}_1)$ holds  with
\begin{align*}
R=R^*,~ C_R= (L_0+ L_1 )(1+R^{\ell_0})^2,~ K_R=L_0+L_1,~   K_R^*=(L_{2 }/2)\wedge L_{4 },
\end{align*}
$\theta=\gamma$ and $ \delta_R=(L_{2}/(2L_{3}))^{\frac{1}{\gamma}}\wedge1$.
Moreover,  under $({\bf A}_0)$ and  $({\bf A}_4)$,  for  $b^{(\delta)}=b$  and  $\pi^{(\delta)} $ defined in \eqref{ET-},
$({\bf H}_1)$ is satisfied  with
\begin{align*}
R=R_*, ~ C_R=2L_0\varphi(R),~K_R=2L_0, ~ K_R^*=L_5,~\theta=\gamma,~\delta_R=\varphi(R)^{-\frac{1}{\gamma}}.
\end{align*}
\end{lemma}

\begin{proof}
 For the case that  $b^{(\delta)}$ is given in \eqref{EY-} and $\pi^{(\delta)}={\mbox Id}$, we take $R=R^*$, $ \theta=\gamma$ and $\delta_R=(L_{2}/(2L_{3}))^{\frac{1}{\gamma}}\wedge1$. Let $B_R=\{x:|x|\le R\}$ and $B_R^c$ be the corresponding  complement.
 From (${\bf A}_0$) and (${\bf A}_2$),
 we infer   that for   any   $x,y\in\R^d $ and $\delta\in(0,1],$
\begin{align*}
\big|b^{(\delta)}(x)-b^{(\delta)}(y)\big|
&\le\frac{ (L_0+L_1)  ( 1+ |x|^{\ell_0}+|y|^{\ell_0} + \delta^\gamma |x|^{\ell_0}|y|^{\ell_0} )|x-y|}{1+\delta^\gamma(|x|^{\ell_0}+|y|^{\ell_0} +\delta^{ \gamma}|x|^{\ell_0} |y|^{\ell_0})} \\
&\le (L_0+L_1)\big((1+ R^{\ell_0})^2\I_{\{x,y\in B_R\}}+\delta^{-\gamma}\I_{\{x\in B_R^c\}\cup\{y\in B_R^c\}}\big)|x-y|,
\end{align*}
where in the second inequality we took advantage of the fact that $[0,\8)\ni r\mapsto (1+r)/(1+\delta^\gamma r)$ is increasing in case of  $\delta\in(0,1]$.
Accordingly, \eqref{E0}  is satisfied  for the setting  $\pi^{(\delta)}=\mbox{Id} $. Next,  $({\bf A}_3)$ enables us to derive that
 for all  $x,y\in\R^d$  with $x\in B_R^c$  or $y\in B_R^c$,
\begin{align*}
\<x-y,b^{(\delta)}(x)-b^{(\delta)}(y)\>
&\le -\frac{\big( (L_{2 }-L_{3 }\delta^\gamma)(1+|x|^{\ell_0}+|y|^{\ell_0}) +L_{4 }\delta^{\gamma}  |x|^{\ell_0}|y|^{\ell_0} \big) |x-y|^2 }{1+\delta^\gamma(|x|^{\ell_0}+|y|^{\ell_0} +\delta^{ \gamma}|x|^{\ell_0} |y|^{\ell_0})}\\
&\le  -\big(  (L_{2 }/2)\wedge L_{4 }\big)|x-y|^2,
\end{align*}
where the second inequality holds true by exploiting  $L_3\delta^\gamma\le L_2/2$ for $\delta\in(0,\delta_R]$
and making use of the fact that $[0,\8)\ni r\mapsto (1+r)/(1+\delta^\gamma r)$, $\delta\in(0,\delta_R]$,  is increasing once more. As a consequence, we conclude that \eqref{E*} is reached.

Concerning the setting that $b^{(\delta)}=b$  and  $\pi^{(\delta)} $ is defined in \eqref{ET-}, we choose $R=R_*$, $  \theta=\gamma$ and $\delta_R=\varphi(R)^{-\frac{1}{\gamma}}$. By means of $({\bf A}_0)$, we deduce that  for any $\delta>0$, and $x,y\in\R^d,$
\begin{align*}
\big|b(\pi^{(\delta)}(x))-b(\pi^{(\delta)}(y))\big|&\le L_0\big(1+(|x|\wedge\varphi^{-1}(\delta^{-\gamma}))^{\ell_0}+(|y|\wedge\varphi^{-1}(\delta^{-\gamma}))^{\ell_0}\big)|x-y|\\
&\le   2L_0\big(\varphi(R)\I_{\{x,y\in B_R\}}+  \delta^{-\gamma}\I_{\{x\in B_R^c\} \cup\{y\in B_R^c\}}\big)|x-y|,
\end{align*}
where in the first inequality we used $|\pi^{(\delta)}(x)|=|x|\wedge\varphi^{-1}(\delta^{-\gamma})$ and the contractive property of $\pi^{(\delta)}$ (see Lemma \ref{lem0}), and in the second inequality we made use of $\varphi(|x|\wedge\varphi^{-1}(\delta^{-\gamma}))\le \delta^{-\gamma}$. Consequently, \eqref{E0}  is examinable. Next,  note that for
 $x\in\R^d$ with $x\in B^c_{R}$ and $\delta\in(0,\delta_R]$,
 \begin{align*}
|\pi^{(\delta)}(x)|=|x|\wedge \varphi^{-1}(\delta^{-\gamma})\ge|x|\wedge \varphi^{-1}(\varphi(R)) \ge R.
\end{align*}
This, together with $({\bf A}_4)$, implies that   for any $x,y\in\R^d$ with $x\in B_{R}^c$ or $y\in B_{R}^c,$
\begin{align*}
\big\< \pi^{(\delta)}(x)-\pi^{(\delta)}(y), b(\pi^{(\delta)}(x))-b(\pi^{(\delta)}(y))\big\>\le-L_5 \big|\pi^{(\delta)}(x)-\pi^{(\delta)}(y)\big|^2.
\end{align*}
Correspondingly,    \eqref{E*}  is also valid.
\end{proof}

\begin{lemma}\label{lem2}
For $\pi^{(\delta)}=\mbox{ Id}$ and $\pi^{(\delta)} $ defined in \eqref{ET-}, the hypothesis  \eqref{EY} holds true with $  \delta_r^*=1, r\ge0, $
and $ \delta_r^*=\varphi(r)^{-\frac{1}{\gamma}}$, $r\ge0, $ respectively. Furthermore, under Assumption $({\bf A}_0)$, for $(i)$ $b^{(\delta)}$ given in \eqref{EY-} and $\pi^{(\delta)}={\mbox Id}$ as well as  $(ii)$ $b^{(\delta)}=b$ and $\pi^{(\delta)} $ defined in \eqref{ET-},
  \eqref{EY*} is fulfilled  for
 \begin{align*}
\psi_1(r)=L_0 (r^{\ell_0}+1)r+|b({\bf 0})| \quad \mbox{ and } \quad  \psi_2(r)\equiv L_0,\quad r\ge0.
 \end{align*}
\end{lemma}

\begin{proof}
Below, we set $ \delta_r^*:=\varphi(r)^{-\frac{1}{\gamma}},r\ge0,$ and
stipulate $\delta\in(0,  \dd_r^*]$.
It is ready to see that \eqref{EY} is valid for the case $\pi^{(\delta)} =\mbox {Id}$. For $\pi^{(\delta)} $ defined in \eqref{ET-},   we have $|\pi^{(\delta)}(x)|=|x|\wedge \varphi^{-1}(\delta^{-\gamma}),x\in\R^d$. Since $\varphi^{-1}(\delta^{-\gamma})\ge r$ for any $\delta\in(0, \dd_r^*]$,
  it follows readily that
 $|\pi^{(\delta)}(x)|=|x| $ for any $x\in\R^d$ with $|x|\le r$, and meanwhile
\begin{align*}
\inf_{|x|\ge r}\big|\pi^{(\delta)}(x)\big|\ge r\wedge \varphi^{-1}(\delta^{-\gamma})=r.
\end{align*}
Consequently, \eqref{EY} is  verifiable for $\pi^{(\delta)}$  in \eqref{ET-}.

In the sequel, we turn to check the validity of \eqref{EY*}.
 For $b^{(\delta)}$ given in \eqref{EY-},
 we deduce from (${\bf A}_0$)  that for any $r\ge0$, $\delta\in(0,1]$ and $x\in\R^d,$
\begin{align*}
|b^{(\delta)}(x)|&\le \frac{|b(x)-b({\bf 0})|}{1+\dd^\gamma|x|^{\ell_0}}+|b({\bf 0})|\\
&\le \ff{L_0 (|x|^{\ell_0}+1)}{1+\dd^\gamma|x|^{\ell_0}}|x|\mathds 1_{\{|x|>r\}}+L_0 (r^{\ell_0 }+1)r+|b({\bf 0})|\\
&\le L_0 \dd^{-\gamma}|x|\mathds 1_{\{|x|>r\}}+L_0 (r^{\ell_0 }+1)r+|b({\bf 0})|,
\end{align*}
in which in the last display we utilized the fact that $[0,\8)\ni(1+r)/(1+\delta^\theta r) $ is increasing.
Therefore, we conclude that \eqref{EY*} is available for $\psi_1(r)=L_0 (r^{\ell_0 }+1)r+|b({\bf 0})|$ and $\psi_2(r)\equiv L_0$.

For $\pi^{(\delta)}$ defined in \eqref{ET-},
we obtain from (${\bf A}_0$) that
for any    $r\ge0,$
\begin{align*}
\big|b(\pi^{(\delta)}(x))\big|
&\le |b({\bf 0})|+L_0\varphi(|\pi^{(\delta)}(x)|)|\pi^{(\delta)}(x)|\\
&\le |b({\bf 0})|+L_0 \varphi(r)r+L_0\varphi(|x|\wedge\varphi^{-1}(\delta^{-\gamma}))|\pi^{(\delta)}(x)|\I_{\{|\pi^{(\delta)}(x)|>r\}}\\
&\le |b({\bf 0})|+L_0\varphi(r)r+ L_0 \dd^{-\gamma}|\pi^{(\delta)}(x)|\I_{\{|\pi^{(\delta)}(x)|>r\}},
\end{align*}
where in the second inequality we used $|\pi^{(\delta)}(x)|=|x|\wedge\varphi^{-1}(\delta^{-\theta})$.  As a result, \eqref{EY*} is also provable for  $b^{(\delta)}=b$ and $\pi^{(\delta)} $ defined in \eqref{ET-}.
\end{proof}

Before the proof of Theorem \ref{thm}, there are still some   warm-up materials to be prepared.
For $\alpha_0:=\frac{K_R^*}{8d}$, $\beta_0:=\frac{96 C_R}{K_R^*}$, and $R>0$  given in $({\bf H}_1)$,   we define the $C^1$-function $[0,\8)\ni r\mapsto \varphi_{\alpha_0,\beta_0}(r)$ as below:
\begin{equation}\label{EW}
\varphi_{\alpha_0,\beta_0}(r)=
\begin{cases}
\alpha_0 \beta_0(1+\beta_0)(1 +R)^2-\alpha_0 \beta_0 r^2,~ \quad\quad 0\le r\le 1 +R,\\
\alpha_0\big(r-(1+\beta_0)(1 +R)\big)^2,\quad  \quad\quad\quad\quad  1 +R<r\le (1+\beta_0)(1 +R),\\
0,\qquad\qquad\qquad \qquad\qquad \qquad\quad \quad\quad~~  r>(1+\beta_0)(1 +R).
\end{cases}
\end{equation}

The following lemma shows that, under Assumption (${\bf H}_2$),  $(X_{n\delta}^\delta)_{n\ge0}$ determined by \eqref{E2} satisfies a Lyapunov type condition in the semigroup form. More precisely, we have the statement below.

\begin{lemma}\label{pro}
Assume that $({\bf H}_{2})$ is  satisfied, and suppose further
\begin{align}\label{EE4-}
\si^2\ge  \frac{8}{d}(1+ R) \big ( \psi_1( R)  \vee (12C_R  \psi_1(R_0)/K_R^* )\big).
\end{align}
Then, for   any   $\delta\in(0,\delta_1^\star]$ and $n\ge0$,
\begin{align}\label{WE-}
\E\big(V(X_{(n+1)\delta}^\delta)\big|\mathscr F_{n\delta}\big)\le V(X_{n\delta}^\delta)-\Big(3C_R\I_{\{|X_{n\delta}^\delta|\le R\}}-\frac{3K_R^*}{8}\I_{\{|X_{n\delta}^\delta|>R\}}\Big)\si^2\delta,
\end{align}
 where the radial function $V$ is defined by $V(x) =\varphi_{\alpha_0,\beta_0}(|x|), x\in\R^d$.
\end{lemma}

\begin{proof}
Since Brownian motions possess  independent increments,  it is sufficient to verify  \eqref{WE-} for  the case  $n=0,$ i.e.,
\begin{align}\label{EW-}
\varphi_{\alpha_0,\beta_0}(|\hat x^\delta+\si W_{ \delta}|)\le\varphi_{\alpha_0,\beta_0}(|x|)-\Big(3C_R\I_{\{|x|\le R\}}-\frac{3}{8} K_R^*\I_{\{|x|>R\}}\Big)\si^2\delta,\quad x\in\R^d,
\end{align}
where $$  \hat x^\delta: =\pi^{(\delta)}(x)+ b^{(\delta)}(\pi^{(\delta)}(x))\delta,\quad x\in\R^d.$$

In the following analysis,  we shall fix $\delta\in(0,   \delta_1^\star]$ so that $|\pi^{(\delta)}(x)|=|x|, x\in\R^d$ (by virtue of  (${\bf H}_2$) with $r=R_0$ therein),
\begin{align}\label{EE3-}
 \psi_2(1)\delta^{1- \theta}\le \frac{1}{2}\quad \mbox{ and } \quad  2 d(2+d )\big(1+ R^2\big)\si^4\delta  \le \frac{1}{2}d\si^2.
\end{align}
Via the mean value theorem, along with $\|\varphi'_{\alpha_0,\beta_0}\|_\8=2\alpha_0\beta_0(1+R)$, we obtain  that  for  any $r\ge0$ and
 $x\in\R^d $ with $|x|\le r\le  R_0$   (given in $({\bf H}_3)$),
\begin{equation}\label{EE1-}
\begin{split}
\varphi_{\alpha_0,\beta_0}(|\hat x^\delta+\si W_{ \delta}|)&\le \varphi_{\alpha_0,\beta_0}(|\pi^{(\delta)}(x)+\si W_{ \delta}|)\\
&\quad+ 2\alpha_0\beta_0(1+R)\big||\hat x^\delta+\si W_{ \delta}|-|\pi^{(\delta)}(x)+\si W_{ \delta}|\big|\\
&\le \varphi_{\alpha_0,\beta_0}(|\pi^{(\delta)}(x)+\si W_{ \delta}|)+ 2\alpha_0\beta_0(1+R)|b^{(\delta)}(\pi^{(\delta)}(x))|\delta\\
&\le \varphi_{\alpha_0,\beta_0}(|\pi^{(\delta)}(x)+\si W_{ \delta}|)\\
&\quad+2\alpha_0\beta_0(1+R)\big(\psi_1(r) +  \psi_2(r ) \delta^{- \theta}|\pi^{(\delta)}(x)|\I_{\{|\pi^{(\delta)}(x)|>r\}}\big)\delta\\
&\le \varphi_{\alpha_0,\beta_0}(|\pi^{(\delta)}(x)+\si W_{ \delta}|)+2\alpha_0\beta_0(1+R)\psi_1(r)\delta,
\end{split}
\end{equation}
where the second inequality is available  due to the triangle inequality,  and the third inequality holds true by making  use of \eqref{EY*}, and  taking advantage of $|\pi^{(\delta)}(x)|=|x|$ for any $x\in\R^d$ with $|x|\le R_0$ and $\delta\in(0,   \delta_1^\star  ]$; see \eqref{EY} for more details.

Recall from \cite[(3.10)]{BV}   that $\R^d\ni x\mapsto f(x):=h(g(x)) $ is  concave, where $h:\R\to\R$ is   concave and non-increasing, and $g:\R^d\to\R$ is convex. Therefore, by invoking the fact  that
 $h(r):=\varphi_{\alpha_0,\beta_0}(r)-\alpha r^2, r\ge0,$ is   concave and non-increasing  and that    $g(z):=|z|,z\in\R^d,$ is convex,
 we infer from Jensen's inequality that for any $z\in\R^d,$
\begin{equation*}
 \E\varphi_{\alpha_0,\beta_0}(|z+\si W_{ \delta}|)  \le \varphi_{\alpha_0,\beta_0}(|z|)-\alpha_0 |z |^2+\alpha_0\E|z+\si W_{ \delta}|^2 =\varphi_{\alpha_0,\beta_0}(|z|)+\alpha_0\si^2d\delta,
\end{equation*}
where in the identity we used the fact that $\E \<z,W_\delta\>=0$ and $\E|W_\delta|^2=d\delta.$ The estimate above,
along with  \eqref{EE1-} and the increasing property of $[0,\8)\ni r\mapsto \psi_1(r)$, implies that
\begin{align*}
 \varphi_{\alpha_0,\beta_0}(|\hat x^\delta+\si W_{ \delta}|)
&\le \big(\varphi_{\alpha_0,\beta_0}(|\pi^{(\delta)}(x)+\si W_{ \delta}|)+2\alpha_0\beta_0(1+ R) \psi_1( R)\delta\big)
  \I_{\{|x|\le    R\}} \\
 &\quad +\big(\varphi_{\alpha_0,\beta_0}(|\pi^{(\delta)}(x)|)+\alpha_0\si^2d\delta+2\alpha_0\beta_0(1+ R) \psi_1(R_0)\delta\big)\I_{\{ R<|x|\le   R_0\}}\\
&\quad+\big(\varphi_{\alpha_0,\beta_0}(|\hat x^\delta|)+\alpha_0\si^2d\delta\big)\I_{\{ R_0<|x| \}}\\
&=:\Lambda_1(x,\delta)+\Lambda_2(x,\delta)+\Lambda_3(x,\delta).
\end{align*}

In terms of the definition of $\varphi_{\alpha_0,\beta_0},$ it follows that for any $r\ge0,$
\begin{equation}\label{EP}
\begin{split}
\varphi_{\alpha_0,\beta_0}(r)
&=\big(\alpha_0\beta_0(1+\beta_0)(1+ R)^2-\alpha_0\beta_0r^2\big) \I_{\{0\le r\le 1+ R\}}  +\varphi_{\alpha_0,\beta_0}(r)\I_{\{  r>1+ R\}}\\
&= \alpha_0\beta_0(1+\beta_0)(1+ R)^2-\alpha_0\beta_0r^2 \\
&\quad +\big(\varphi_{\alpha_0,\beta_0}(r)-\alpha_0\beta_0(1+\beta_0)(1+ R)^2+\alpha_0\beta_0r^2\big)\I_{\{  r>1+R\}}\\
&\le \alpha_0\beta_0(1+\beta_0)(1+ R)^2-\alpha_0\beta_0r^2+\alpha_0\beta_0r^2\I_{\{  r>1+ R\}},
\end{split}
\end{equation}
where in the inequality we used the fact that $\varphi_{\alpha_0,\beta_0}(r)-\alpha_0\beta_0(1+\beta_0)(1+R)^2\le0$ by virtue of
$\|\varphi_{\alpha_0,\beta_0}\|_\8=\alpha_0\beta_0(1+\beta_0)(1+R)^2$. For notation brevity, we  set
 $$
h(x): =\E\big(  |\pi^{(\delta)}(x)+\si W_{ \delta}|^2  \I_{\{|\pi^{(\delta)}(x)+\si W_{ \delta}|>1+R\}}\big),\quad x\in\R^d.$$
 Subsequently,  by invoking  the  estimate  \eqref{EP},
  we derive that for any  $x\in\R^d$ with $|x|\le  R,$
\begin{equation}\label{EE*}
\begin{split}
\E\varphi_{\alpha_0,\beta_0}(|\pi^{(\delta)}(x)+\si W_{ \delta}|)
&\le \alpha_0\beta_0(1+\beta_0)(1+ R)^2-\alpha_0\beta_0  \E|\pi^{(\delta)}(x)+\si W_\delta|^2+\alpha_0 \beta_0h(x)
\\
&=\alpha_0\beta_0(1+\beta_0)(1+ R)^2-\alpha_0\beta_0\big(|\pi^{(\delta)}(x)|^2+\si^2d\delta\big)  +\alpha_0 \beta_0h(x)
\\
&= \varphi_{\alpha_0,\beta_0}(|x|)-d\alpha_0\beta_0  \si^2\delta +\alpha_0 \beta_0h(x),
\end{split}
\end{equation}
where in the first identity we used the fact that $\E\<\pi^{(\dd)}(x), W_{ \delta}\>=0$ and   $\E|W_\dd|^2=d\delta$, and in the second identity we  utilized  $|\pi^{(\delta)}(x)|=|x|$ for any $x\in\R^d$ with $|x|\le R_0$ and $\delta\in(0, \delta_{R_0}^*]$  followed by taking advantage of the definition of    $\varphi_{\aa_0,\beta_0}$ once again.
By means of  $|\pi^{(\delta)}(x)|\le|x|$,
one apparently has
 for any $x\in\R^d$ with $|x|\le R,$
\begin{align*}
\{|\pi^{(\delta)}(x)+\si W_{ \delta}|\ge1+R\}\subseteq \{  |  W_{ \delta}|\ge 1/|\si|\}.
\end{align*}
This, together with \eqref{EY}, $|\pi^{(\delta)}(x)|\le|x|$ as well as the Chebyshev inequality,    implies that  for any $x\in\R^d$ with $|x|\le  R,$
\begin{align*}
h(x)
&\le2\big(R^2\P(|W_\delta|>1/|\si|)  +\si^2\E\big(| W_{ \delta}|^2  \I_{\{|    W_{ \delta}|> 1/|\si| \}}\big)\big)\\
&\le  2(1+R^2)\si^4\delta^2 \E|W_1|^4\\
&=2(1+R^2)d(2+d )\si^4\delta^2,
\end{align*}
where in the first inequality and in the  identity we employed the scaling property of $(W_t)_{t\ge0}$, and
  $\E |W_1|^4=2d(1+d/2),$ separately. Consequently,   by invoking  \eqref{EE3-},
  we obtain from \eqref{EE*}  that
  for any $x\in\R^d$ with $|x|\le  R,$
\begin{equation}\label{EW2}
\begin{split}
\Lambda_1(x,\delta)&\le \varphi_{\alpha_0,\beta_0}(|x|)-\alpha_0\beta_0\big (d\si^2  -  2 d(2+d )\big(1+ R^2 \big)\si^4\delta-2(1+R)\psi_1(R)  \big )\delta\\
&\le  \varphi_{\alpha_0,\beta_0}(|x|)-\frac{1}{2}\alpha_0\beta_0\big ( d    -4(1+R)\psi_1(R)/\si^2 \big)\si^2\delta\\
&\le  \varphi_{\alpha_0,\beta_0}(|x|)-\frac{1}{4}\alpha_0\beta_0  d\si^2\delta\\
&=   \varphi_{\alpha_0,\beta_0}(|x|)-3C_R\si^2\delta,
\end{split}
\end{equation}
 where the  third inequality is provable   thanks to \eqref{EE4-}, and the identity is attainable by taking the definitions of $\alpha_0,\beta_0$ into consideration.

Notice that $\alpha_0\beta_0=12C_R/d$ and $   2\alpha_0\beta_0(1+R) \psi_1(R_0)/\si^2 \le K_R^*/4$ by taking  \eqref{EE4-} into consideration. Thus,
 for any $x\in\R^d$ with $ R<|x|\le   R_0$, it follows from $|\pi^{(\delta)}(x)|=|x|$ that
\begin{equation}\label{EW3}
\begin{split}
\Lambda_2(x,\delta)&\le \varphi_{\alpha_0,\beta_0}(|x|)+\bigg(\alpha_0 d +\frac{2\alpha_0\beta_0(1+ R) \psi_1(R_0) }{\si^2}\bigg)\si^2\delta \\ &\le  \varphi_{\aa_0,\bb_0}(|x|)+\frac{3}{8} K_R^*\si^2\delta.
\end{split}
\end{equation}

From \eqref{EY*} with $r=1$ therein,
it is easy to see that for any $x\in\R^d$ and $r\ge0,$
\begin{align*}
|\hat x^\delta|\ge|\pi^{(\delta)}(x)|-| b^{(\delta)}(\pi^{(\delta)}(x))|\delta
&\ge\big(1 -\psi_2(1)\delta^{1- \theta}\big)|\pi^{(\delta)}(x)|  -\psi_1(1).
\end{align*}
Hence, \eqref{EY} and \eqref{EE3-}  imply that for any $x\in\R^d$ with $|x|\ge  R_0:=2((1+ R)(1+\beta_0)+\psi_1(1)),$
\begin{align*}
|\hat x^\delta| \ge\frac{1}{2}|\pi^{(\delta)}(x)|  -\psi_1(1)\ge\frac{1}{2}  R_0-\psi_1(1)=(1+R)(1+\beta_0).
\end{align*}
 Whereafter,  definitions of $\alpha$ and $\varphi_{\alpha_0,\beta_0}$ enable  us to derive that for any $x\in\R^d$ with $|x|>R_0,$
\begin{align}\label{EW4}
\Lambda_3(x,\delta)\le \frac{1}{8}K_R^*\si^2\delta.
\end{align}

At length, the desired assertion \eqref{EW-} follows from the estimates \eqref{EW2}, \eqref{EW3}, and \eqref{EW4}.
\end{proof}

With the aid of Proposition \ref{pro}, we move on to complete the proof of Theorem \ref{thm}.
\begin{proof}[Proof of Theorem \ref{thm}]
Below, we fix $\delta\in(0, \delta_1^\star  ]$ so that
\begin{align}\label{E9}
1-2K_R^*\delta\ge0 \quad \mbox{ and } \quad  K_R^2\delta^{2(1- \theta)}\le K_R^*\delta.
\end{align}
Note trivially that for any $n\in\mathbb N,$ $\mu,\nu\in\mathscr P_2(\R^d)$ and $\pi\in\mathscr C(\mu,\nu)$,
\begin{align*}
\mathbb W_2\big(\mu P^{(\delta)}_{n\delta},\nu P^{(\delta)}_{n\delta}\big)^2 \le \int_{\R^d\times \R^d}\mathbb W_2\big(\delta_x P^{(\delta)}_{n\delta},\delta_y P^{(\delta)}_{n\delta}\big)^2\pi(\d x,\d y).
\end{align*}
Provided that there exist  constants $C,\ll>0$ such that for all  $x,y\in\R^d,$
\begin{align}\label{E4}
 \E\big|X_{n\delta}^{\delta,x}-X_{n\delta}^{\delta,y}\big|^2\le C\e^{-\ll n\dd}|x-y|^2,
\end{align}
the assertion \eqref{E3} follows directly by taking advantage of  the basic fact:
\begin{align*}
\mathbb W_2\big(\delta_x P^{(\delta)}_{n\delta},\delta_y P^{(\delta)}_{n\delta}\big)^2\le \E\big|X_{n\delta}^{\delta,x}-X_{n\delta}^{\delta,y}\big|^2.
\end{align*}
Since the modified drift $b^{(\delta)}$ is not globally dissipative, it is impossible  to verify  \eqref{E4} directly via the classical synchronous coupling. To handle this issue, we introduce  the following distance-like function:  for any $x,y\in\R^d,$
\begin{align*}
\rho(x,y)=|x-y|^2\big(\si^2+V(x)+V(y)\big),
\end{align*}
where $\si$ is the noise intensity given in the SDE \eqref{E1}, and $V(x):=\varphi_{\alpha_0,\beta_0}(|x|), x\in\R^d, $  defined  in Lemma  \ref{pro}. Due to the alternatives of $\alpha_0,\beta_0$, it holds that
\begin{align}\label{WQ2}
\|V\|_\8 =\frac{12 C_R}{d}\Big(1+\frac{96C_R}{K_R^*}\Big)(1+R)^2.
\end{align}
 Therefore, the function  $V$
is uniformly bounded and  the quasi-distance $\rho(x,y)$ defined above is comparable to  $|x-y|^2$.  More precisely, we have
 \begin{align}\label{^7}
	\si^2|x-y|^2\le \rho(x,y)\le |x-y|^2(\si^2+2\|V\|_\8).
\end{align}
Based on this fact, along with the fundamental  inequality: $r^a\le \e^{a(r-1)}$ for all $r,a>0,$
it is sufficient to show that  for  $\delta\in(0,  \delta_1^\star  ]$ and $x,y\in\R^d,$
\begin{align}\label{E14}
\E\rho(X_{n\delta}^{\delta,x}, X_{n\delta}^{\delta,y})\le \big(1-(\lambda_1\wedge\lambda_2)  \delta\big)^n\rho(x,y)
\end{align}
in order ro achieve \eqref{E4}, where  $\lambda_1,\lambda_2>0$ were   defined in \eqref{WQ*}.

Obviously, we have
 for any $z_1,z_2\in\R^d,$
\begin{align*}
\Lambda(z_1,z_2):&=\big|\pi^{(\delta)}(z_1)-\pi^{(\delta)}(z_2)
+b^{(\delta)}(\pi^{(\delta)}(z_1))\delta-b^{(\delta)}(\pi^{(\delta)}(z_2))\delta\big|^2\\
&=\big|\pi^{(\delta)}(z_1)-\pi^{(\delta)}(z_2)\big|^2\\
&\quad+2\<\pi^{(\delta)}(z_1)-\pi^{(\delta)}(z_2),b^{(\delta)}(\pi^{(\delta)}(z_1)) -b^{(\delta)}(\pi^{(\delta)}(z_2))\>\delta\\
&\quad+\big|b^{(\delta)}(\pi^{(\delta)}(z_1)) -b^{(\delta)}(\pi^{(\delta)}(z_2)) \big|^2\delta^2.
\end{align*}
This, together with $\pi^{(\delta)}({\bf0})={\bf0}$, the contractive property of $\pi^{(\delta)}$ as well as
(${\bf H}_1$), yields that for any $z_1,z_2\in\R^d $ with $|z_1|\le R$ and $ |z_2|\le R$,
\begin{align}\label{E6}
\Lambda(z_1,z_2)\le(1+C_R\delta)^2 |z_1-z_2|^2,
\end{align}
and that for any $z_1,z_2\in\R^d $ with $|z_1|> R$ or  $ |z_2|> R$,
\begin{equation}\label{E7}
\begin{split}
\Lambda(z_1,z_2)&\le (1-2K_R^*\delta) \big|\pi^{(\delta)}(z_1)-\pi^{(\delta)}(z_2)\big|^2+K_R^2|z_1-z_2|^2\delta^{2(1-\theta)}\\
&\le \big(1-2K_R^*\delta +K_R^2 \delta^{2(1- \theta)}\big) |z_1-z_2|^2\\
&\le \big(1- K_R^*\delta  \big) |z_1-z_2|^2,
\end{split}
\end{equation}
where the second inequality and the third inequality hold true owing to \eqref{E9}.
Next,
according to the definition of $\rho,$ we find from Lemma  \ref{pro} that
\begin{equation}\label{E13}
\begin{split}
\E \rho(X_{(n+1)\delta}^{\delta,x}, X_{(n+1)\delta}^{\delta,y})
&=\E\big(\Lambda(X_{n\delta}^{\delta,x},X_{n\delta}^{\delta,y}) \big(\si^2+\E\big (V(X_{(n+1)\delta}^{\delta,x})\big|\mathscr F_{n\delta}\big)\\  &\quad\quad\quad+\E\big (V(X_{(n+1)\delta}^{\delta,y})\big|\mathscr F_{n\delta}\big)\big)\big)\\
&\le \E\Gamma(X_{n\delta}^{\delta,x},X_{n\delta}^{\delta,y}),
\end{split}
\end{equation}
where for any $z_1,z_2\in\R^d,$
\begin{align*}
\Gamma(z_1,z_2):&=\Lambda(z_1,z_2) \Big(\si^2+ V(z_1)+ V(z_2)- 3C_R \si^2\delta\big(\I_{\{|z_1|\le R\}}+\I_{\{|z_2|\le R\}}\big)\\
&\qquad\qquad\qquad+ \frac{3}{8} \si^2 K_R^*\delta\big(\I_{\{|z_1|> R\}}+\I_{\{|z_2|> R\}}\big)\Big).
\end{align*}

Making use of \eqref{E9} and $C_R\delta\le1$ for $\delta\in(0,  {\delta_1^\star }  ]$ enables us to derive from \eqref{WQ2} and \eqref{E6} that for any $z_1,z_2\in\R^d $ with $|z_1|\le R$ and $ |z_2|\le R$
\begin{equation}\label{E11}
\begin{split}
\Gamma(z_1,z_2)
  &\le (1+C_R\delta)^2 \bigg( 1-\frac{ 6 C_R \si^2\delta }{\si^2+V(z_1)+V(z_2)}\bigg)\rho(z_1,z_2)\\
  &\le (1+3C_R \delta)  \bigg( 1-\frac{  6 C_R \si^2\delta }{\si^2+2\|V\|_\8}\bigg)\rho(z_1,z_2)\\
  &\le (1-\lambda_1\delta)\rho(z_1,z_2),
\end{split}
\end{equation}
where $\lambda_1>0$ was defined in \eqref{WQ*}.
On the other hand, by invoking   \eqref{E7} and $1-K_R^*\delta>0$ (see \eqref{E9}),  we derive that  for   $z_1,z_2\in\R^d $ with $|z_1|> R$ or $ |z_2|> R$,
\begin{equation}\label{E12}
\begin{split}
\Gamma(z_1,z_2)
&\le  ( 1- K_R^*\delta  ) \bigg(1+ \frac{3K_R^*\si^2\delta}{4(\si^2+ V(z_1)+ V(z_2))}\bigg)\rho(z_1,z_2)\\
&\le ( 1- K_R^*\delta  ) (1+  3K_R^* \delta/4  )\rho(z_1,z_2)\\
&\le  (1-\lambda_2\delta )\rho(z_1,z_2),
\end{split}
\end{equation}
where $\lambda_2:=\frac{1}{4}K_R^*$. Next, by taking \eqref{E11} and \eqref{E12} into consideration, we obtain from \eqref{E13} that
\begin{equation*}
\E \rho(X_{(n+1)\delta}^{\delta,x}, X_{(n+1)\delta}^{\delta,y})\le\big(1-(\lambda_1\wedge\lambda_2)\delta \big) \E\rho(X_{n\delta}^{\delta,x}, X_{n\delta}^{\delta,y}),
\end{equation*}
where the prefactor $1-(\lambda_1\wedge\lambda_2)\delta \ge0$ for any $\delta\in(0, \delta_1^\star].$ Whence,  \eqref{E14} follows from \eqref{^7} and
 an inductive argument.
\end{proof}

Based on Lemmas \ref{lem1} and \ref{lem2}, as well as Theorem \ref{thm},
 we are in position to carry out the proof of Theorem \ref{thm-0}.

\begin{proof}[Proof of Theorem \ref{thm-0}]
With the aid of Theorem \ref{thm}, as far as  the TEM scheme \eqref{EW0-} and the PEM algorithm \eqref{EWW} are concerned,
it is sufficient to verify the prerequisites required in  Theorem \ref{thm}
in order to complete the proof of Theorem \ref{thm-0}, respectively.

According to Lemma \ref{lem0}, Assumption $({\bf A}_1)$ is satisfied the TEM scheme \eqref{EW0-} and the PEM algorithm \eqref{EWW}.
Regarding the TEM \eqref{EW0-}, according to Lemmas \ref{lem1} and \ref{lem2},
   $({\bf H}_1)$ and $({\bf H}_2)$ hold  with
\begin{align}\label{QW5}
R=R^*,~ C_R= (L_0+ L_1 )(1+R^{\ell_0})^2,~  K_R=L_0+L_1,~   K_R^*=(L_{2 }/2)\wedge L_{4 }, ~ \theta=\gamma,
\end{align}
and
\begin{align}\label{QW}
\psi_1(r)=L_0 (r^{\ell_0}+1)r+|b({\bf 0})| \quad \mbox{ and } \quad  \psi_2(r)\equiv L_0,\quad r\ge0.
 \end{align}
Subsequently, in case of $|\si|\ge\si_0$ (given in \eqref{DD}), Theorem \ref{thm} implies that  \eqref{RT*1} holds for positive  constants
$\lambda,C_0$ introduced in \eqref{DD} and all  $\delta \in(0,\delta_1^\star]$, where  $\delta_1^\star$ was defined in \eqref{dd1}
with the associated quantities being given in \eqref{QW5} and \eqref{QW}.

 As regards the PEM scheme  \eqref{EWW}, by applying  Lemmas \ref{lem1} and \ref{lem2} once more,  we deduce that
 $({\bf H}_1)$ and $({\bf H}_2)$ are fulfilled  for
 \begin{align}\label{QW4}
R=R_*, ~ C_R=2L_0\varphi(R),~K_R=2L_0, ~ K_R^*=L_5,~ \theta=\gamma
\end{align}
 and $\psi_1,\psi_2$ given in \eqref{QW}.  Consequently, as long as $|\si|\ge\si_0$,
 Theorem \ref{thm} enables us to deduce that   the statement  \eqref{RT*1} is also available
 for
 constants $\lambda,C_0>0$ given in \eqref{DD} and all    $\delta\in(0,\delta_1^\star]$,  where
the corresponding quantities are stipulated in \eqref{QW4}.
\end{proof}

\section{Criteria on non-asymptotic $L^2$-Wasserstein bounds and proof of Theorem \ref{cor}}\label{sec3}
 In this part, we follow  a similar routine adopted in Section \ref{sec2} to finish the proof of Theorem \ref{cor}. In the first place, we provide  general criteria to investigate non-asymptotic $L^2$-Wasserstein bounds (see Theorem \ref{IPM} below)
 associated with the modified EM scheme \eqref{E2}. After that, the proof of Theorem \ref{cor} can be finished as an application of Theorem \ref{IPM}.

To furnish
non-asymptotic   $L^2$-Wasserstein bounds between the exact  IPM  and
the associated numerical distribution associated with \eqref{E2}, some additional  conditions  on $b$ and $b^{(\delta)}$ need to be provided. In detail, we  assume that
\begin{enumerate}
\item[$({\bf H}_4)$] there exist constants $\ll_b,C_b >0$ such that for   $x,y\in\R^d$,
 \begin{equation*}
\<x,b(x)\>\le -\ll_b|x|^2+C_b;
\end{equation*}

\item[$({\bf H}_5)$] for any $\aa>0$,  there exist constants $c_\aa ,l_\aa\ge0$ such that for all $\delta\in(0,1]$ and $x\in\R^d$,
\begin{align}\label{EW1-}
|\pi^{(\delta)}(x)-x|\le c_\aa \big(1+|x|^{l_\aa }\big)\delta^\aa;
\end{align}
 moreover, there exist constants ${\beta}>0,l_{\beta}^*,\lambda_{b}^*\ge0$ such that for any $\delta\in(0,1] $ and $x\in\R^d$,
 \begin{align}\label{EW2-}
 |b^{(\delta)}(x)-b(x)|\le \lambda_{b}^*\big(1+|x|^{l_{\beta}^*}\big)\delta^{\beta}.
 \end{align}
\end{enumerate}

Before we proceed, we make some comments on Assumptions $({\bf H}_4)$ and  $({\bf H}_5)$.
\begin{remark}
Under $({\bf A}_0)$ and $({\bf H}_4)$,  \eqref{E1} is strongly well-posed and has   finite $p$-th ($p>0$) moment  in an infinite horizon
as claimed in Lemma \ref{lemma1} below.
 Apparently, $({\bf H}_4)$ holds true trivially for  $\pi^{(\delta)} ={\rm Id}$ and $b^{(\delta)} =b $ (which corresponds to the classical EM scheme). Moreover, we would like to mention that a similar  counterpart   of
 \eqref{EW2-} was imposed in \cite[Assumption (A1.)]{BDMS} to investigate ergodicity and convergence analysis for the tamed unadjusted Langevin algorithms.
\end{remark}

Below, we shall present  a lemma, which also lays a foundation for the proof of Theorem \ref{cor},
 to demonstrate that $({\bf H}_5)$ is satisfied in some scenarios.

\begin{lemma}\label{lem3}
For $\pi^{(\delta)}$ defined  in \eqref{ET-}, \eqref{EW1-} is fulfilled for the pair $(c_\alpha,l_\alpha)= ( 2^{1\vee\frac{\alpha}{\gamma}},1+\ell_0 \alpha/\gamma)$ with $\gamma\in(0,1/2)$.  Under $({\bf A}_0)$,  concerning  $b^{(\delta)}$ and $\bar b^{(\delta)}$ defined respectively in \eqref{EY-} and \eqref{TY}, \eqref{EW2-} holds true respectively with  the triples
\begin{align}\label{DD4}
(\lambda_b^*,l_\beta^*,\beta)=(  3 (L_0\vee|b({\bf0})|), 1+2\ell_0  ,\gamma)~\mbox{ and } ~(\lambda_b^*,l_\beta^*,\beta)=(2(L_0\vee|b({\bf0})|), 1+3\ell_0 ,1).
\end{align}

\end{lemma}

\begin{proof}
In terms of the definition of $\pi^{(\dd)}$ given in \eqref{ET-} and by recalling $\varphi(r)=1+r^{\ell_0},r\ge0$, we find that for all $\alpha>0 $ and $x\in\R^d,$
\begin{align*}
|\pi^{(\delta)}(x)-x|&=|\pi^{(\delta)}(x)-x|\I_{\{|x|>\varphi^{-1}(\delta^{-\gamma}) \}}\\
 &\le \delta^\aa\big(|x|-|\pi^{(\delta)}(x)|)|\varphi( |x|)^{\frac{\aa}{\gamma}} \I_{\{\varphi(|x|)>\delta^{-\gamma}  \}}\\
 &\le 2^{1\vee\frac{\alpha}{\gamma}}\big(1+|x|^{1+\ell_0 \alpha/\gamma }\big)\delta^\aa.
\end{align*}
Thereby, we conclude that \eqref{EW1-} is true for  the pair  $(c_\alpha,l_\alpha)= (2^{1\vee\frac{\alpha}{\gamma}},1+\ell_0 \alpha/\gamma) .$

On the one hand,  for $b^{(\delta)}$ introduced in \eqref{EY-}, we deduce from $({\bf A}_0)$ that for all $x\in\R^d,$
\begin{align*}
\big|b^{(\delta)}(x)-b(x)\big| = \frac{ \delta^\gamma|x|^{\ell_0}|b(x)|}{1+\delta^\gamma|x|^{\ell_0}} &\le \frac{ (L_0(1+|x|^{\ell_0})|x|+|b({\bf0})|)|x|^{\ell_0} \delta^\gamma }{1+\delta^\gamma|x|^{\ell_0}}\\
&\le  3 (L_0\vee|b({\bf0})|)(1+|x|^{1+2\ell_0})  \delta^\gamma.
\end{align*}
On the other hand, regarding $\bar b^{(\delta)}$ defined in \eqref{TY}, we obtain from $({\bf A}_0)$ that
\begin{align*}
\big|\bar b^{(\delta)}(x)-b(x)\big| = \frac{((1+\delta^{}|x|^{2\ell_0})^{1/2}-1)|b(x)|}{(1+\delta^{}|x|^{2\ell_0})^{1/2}}
 &\le \frac{|x|^{2\ell_0}|b(x)|\delta^{}}{2(1+\delta^{}|x|^{2\ell_0})^{1/2}}\\
 &\le2(L_0\vee|b({\bf0})|)(1+|x|^{1+3\ell_0})  \delta^{},
\end{align*}
where in the first inequality we used the basic inequality: $(1+a)^{\frac{1}{2}}-1\le  a/2, a\ge0$.
Based on the previous analysis, \eqref{EW2-} is available for the respective triples defined in \eqref{DD4}.
\end{proof}

The following theorem provides a quantitative bound between the exact IPM  and the corresponding numerical counterpart
associated with the modified EM scheme \eqref{E2} under appropriate conditions, which in turn yields
a   non-asymptotic  $L^2$-Wasserstein bound between the exact IPM  and the corresponding numerical distribution.
\begin{theorem}\label{IPM}	
Assume   $({\bf A}_0)$, $({\bf A}_1)$ and  $({\bf H}_1)$-$({\bf H}_5)$, and suppose further $({\bf A}_5)$ once $\beta $ involved  in $({\bf H}_5)$
satisfies  $\beta> 1/2.$
Then,  there exists a   constant
$  C_0  >0$
such that for all
 $\dd\in(0,\delta_2^{\star\star}]$, $n\ge0,$ $\mu\in \mathscr P_2(\R^d)$,
\begin{equation}\label{E6-1}
\W_{2}\big(\pi_\8,\pi_\8^{(\delta)}\big)\le  C_0 \delta^{1\wedge\beta}  d^{\ff{1}2\ell_\star}   ,
\end{equation}
  and
\begin{align}\label{E6-2}
\W_2(\mu P^{(\dd)}_{n\dd},\pi_\8)\le C_0 \big(\e^{-\ff{1}2(\ll_1\wedge\ll_2) n\dd}	\W_2(\mu,\pi_\8)+\dd^{1\wedge \bb}  d^{\ff{1}2\ell^\star}   \big),
\end{align}
where  $\pi_\8$ $($resp. $\pi_\8^{(\delta)}$$)$
is  the $($resp. unique$)$   IPM of $(X_t)_{t\ge0}$ $($resp. $(X_{n\delta}^\delta)_{n\ge0}$ determined by \eqref{E2}$)$,
\begin{align}\label{DD5}
\ell_\star:= \big(2\ell_0+1+(1+\ell_0^\star)\mathds 1_{\{\bb>\ff12\}}\big)\vee(\ell_0+l_2)\vee l_\beta^{*},~~   ~~ \dd_2^{\star\star}:=\delta_1^\star\wedge\big( K_R^*/(4\psi^2_2(1))\big)^{\frac{1}{1-2 \theta}},
\end{align}
and the quantities $\lambda_1,\lambda_2, \delta_1^\star>0$ were given in Theorem \ref{thm}.
\end{theorem}

Before the proof of Theorem \ref{IPM}, we prepare for several preliminary lemmas.

\begin{lemma} \label{lemma1}
Assume that  $({\bf A}_0)$ and $({\bf H}_4)$ hold.
Then, for any $p>0,$   $t\ge0$ and $x\in\R^d,$
\begin{align}\label{WW2}
\E |X_t^x|^p \le 2C(p)/(p\lambda_b)
+2^{(p/2-1)^+}\e^{-\frac{1}{2}p\lambda_bt}(1+  |x|^p),
\end{align}
where $(X_t^x)_{t\ge0}$ stands for the unique solution to \eqref{E1} with the initial value $X_0=x\in\R^d$, and
\begin{align}\label{PP1}
C(p):=2^{\frac{1}{2}p}\big(  (1-2/p)^+ /{\lambda_b}\big)^{ (p/2-1)^+ }\big(\lambda_b+C_b+\si^2(d+(p-2)^+)/2\big)^{1\vee\frac{p}{2}}.
\end{align}
Consequently,  $(X_t)_{t\ge0}$ admits an   IPM  $\pi_\infty$ satisfying
\begin{align}\label{EW2*}
\pi_\8(|\cdot|^p)\le  2C(p)/(p\lambda_b).
\end{align}
\end{lemma}

\begin{proof}
Under $({\bf A}_0)$ and $({\bf H}_4)$, it is quite standard that the SDE \eqref{E1} is strongly well-posed.
The proof of \eqref{WW2} is more or less standard. Nevertheless, we herein give a sketch    to make the content self-contained and most importantly highlight the dimension dependency  of the associated constant. For the Lyapunov function $V_p(x):=(1+|x|^2)^{\frac{p}{2}}, x\in\R^d,$ it is easy to see that
\begin{align*}
\nn V_p(x)=p(1+|x|^2)^{\frac{p}{2}-1}x~ \mbox{ and } ~\nn^2V_p(x)=p(1+|x|^2)^{\frac{p}{2}-1}I_d+p(p-2)(1+|x|^2)^{\frac{p}{2}-2}x^\top x,
\end{align*}
where $I_d$ means the $d\times d$ identity matrix, and $x^\top$ denotes the transpose of $x.$
Let $\mathcal L$ be the infinitesimal generator of $(X_t)_{t\ge0}$. Thus, we deduce from $({\bf H}_4)$ and Young's inequality that
\begin{equation}\label{WW3}
\begin{split}
(\mathcal LV_p)(x)
&\le p(1+|x|^2)^{\frac{p}{2}-1}\< x,b(x)\>+\ff{1}2\si^2 p(1+|x|^2)^{\frac{p}{2}-1} (d+(p-2)^+ )\\
&\le -p\lambda_bV_p(x)+p\big(\lambda_b+C_b+\si^2(d+(p-2)^+)/2\big)(1+|x|^2)^{\frac{p}{2}-1}\\
&\le -\frac{1}{2}p\lambda_bV_p(x)+C(p),
\end{split}
\end{equation}
where the constant $C(p)$ was defined in \eqref{PP1}.
Next, for any $n\ge1$, define the stopping time
\begin{align*}
\tau_n=\inf\big\{t\ge0: |X_t^x|\ge n\big\}.
\end{align*}
Thus,   It\^o's  formula  and \eqref{WW3} enable  us to derive that
\begin{align*}
\E\Big(\e^{\frac{1}{2}p\lambda_b(t\wedge\tau_n)}V_{p}(X_{t\wedge\tau_n}^x) \Big)\le  V_{p}(x)+ 2  C (p) \e^{\frac{1}{2}p\lambda_bt}/{(p\lambda_b)}.
\end{align*}
This further implies $\tau_n\to\8$ a.s. Whereafter,
the statement \eqref{WW2} can be available by making use of Fatou's lemma.

By keeping \eqref{WW2} in mind, the Krylov-Bogoliubov theorem   (see e.g. \cite[Corollary 3.1.2]{DZ})
implies  that $(X_t)_{t\ge0}$ has an IPM, denoted by  $\pi_\8 $.
Notice that, for any $K>0$, the function $[0,\8)\ni r\mapsto K\wedge r$ is   concave. So, by means of  the invariance of $\pi_\8$, we obtain from Jensen's inequality that
\begin{align*}
\pi_\8 \big(K\wedge|\cdot|^p\big)=\pi_\8 \big(\E\big(K\wedge|X_t^\cdot|^p\big)\big)\le \pi_\8 \big( K\wedge\E|X_t^\cdot|^p \big).
\end{align*}
This, along  with \eqref{WW2} and the basic inequality: $a\wedge (b+c)\le a\wedge b+a\wedge c$ for $a,b,c\ge0,$ leads to
\begin{align*}
\pi_\8 \big(K\wedge|\cdot|^p\big) \le 2C(p)/(p\lambda_b)+\pi_\8 \big( K\wedge\big(2^{(p/2-1)^+}\e^{-\frac{1}{2}p\lambda_bt}(1+|\cdot|^p)\big) \big).
\end{align*}
Afterward, the dominated convergence theorem   and Fatou's lemma  yield  the assertion \eqref{EW2*} by   approaching $K\uparrow\8$ and sending $t\uparrow\8$ successively.
\end{proof}

Below, we provide some sufficient conditions to guarantee that the modified Euler scheme \eqref{E2} has finite moment  in an infinite horizon.

\begin{lemma}\label{lemma3}
Assume   \eqref{EY*} and
 suppose further that  there exist constants $ \lambda_b^\star, C_b^\star>0$ such that for any $\delta\in[0,1] $ and $x\in\R^d,$
\begin{align}\label{WW*}
\<\pi^{(\delta)}(x),b^{(\delta)}(\pi^{(\delta)}(x))\>\le   C_b^\star- \lambda_b^\star|\pi^{(\delta)}(x)|^2.
\end{align}
 Then, for any $p>0$,
 there exists a constant  $  C^*(p)>0$ such that for all $\dd\in (0,\dd_\star]$ and $n\ge0,$
\begin{equation}\label{HB0}
\E\big(|X^{\dd}_{n\dd}|^{p}\big|\mathscr F_0\big)\le  C^*(p)(1+   d^{\ff p2}) +  \e^{-\frac{p\lambda_b^\star  n\dd}{4(3+\lfloor p\rfloor)} }|X^\dd_0|^p,
\end{equation}
in which $\delta_\star:=\big( \lambda_b^\star/(2\psi^2_2(1))\big)^{\frac{1}{1-2\theta}}\wedge(1/\lambda_b^\star)\wedge 1$ with $\psi_2(\cdot)$ being introduced in \eqref{EY*}.
Consequently,  for each fixed $\delta\in(0,\delta_\star],$  $(X_{n\delta}^{\delta})_{n\ge0}$ possesses an   IPM  $\pi^{(\delta)}_\infty$  satisfying
\begin{align}\label{PP}
\pi_\8^{(\delta)}(|\cdot|^p)\le  C^*(p)(1+   d^{\ff p2}).
\end{align}

\end{lemma}

\begin{proof}
Below, we shall stipulate $\delta\in(0,\delta_\star]$ so that
\begin{align}\label{EW3-}
2\psi_2(1)^2  \delta^{ 2(1- \theta)}\le  \lambda_b^\star\delta\quad \mbox{ and } \quad 0\le 1- \lambda_b^\star\dd\le1.
\end{align}
Provided that,  for any
integer   $q\ge3,$
  there exists a constant  $c_1(q) >0$ such that for any $n\ge0,$
\begin{equation}\label{E8}
\E\big(|X^{\dd}_{n\dd}|^{2q}\big|\mathscr F_0\big)\le c_1(q)(1+d^q)+ \e^{- \frac{1}{2}\lambda_b^\star  n\dd }|X^\dd_0|^{2q},
\end{equation}
 H\"older's  inequality implies   that for some constant  $c_2(p) >0,$
 \begin{align*}
\E\big(|X^{\dd}_{n\dd}|^{p}\big|\mathscr F_0\big)
&\le \Big(\E\big(|X^{\dd}_{n\dd}|^{2(3+\lfloor p\rfloor)}\big|\mathscr F_0\big)\Big)^{\frac{p}{2(3+\lfloor p\rfloor)}}
\\&
\le c_2(p)(1+d^{\frac{p}{2}})+ \e^{- \frac{p \lambda_b^\star  n\dd}{4(3+\lfloor p\rfloor)} }|X^\dd_0|^{p}.
\end{align*}
Whence, \eqref{HB0} can be verifiable based on the establishment of \eqref{E8}. Thus, for each fixed $\delta\in(0,\delta_\star]$,  the discrete-time Markov chain $(X_{n\delta}^\delta)_{n\ge0}$ has an IPM, written as  $\pi_\8^{(\delta)}$,  by making use of the Krylov-Bogoliubov theorem (see e.g. \cite[Corollary 3.1.2]{DZ}). In the end, \eqref{PP} can be attainable by following exactly the strategy  to derive \eqref{EW2*}.

Below, we focus on  proving  the statement \eqref{E8}.
In terms of  the scheme \eqref{E2},  we have
\begin{equation*}
\begin{split}
\big|X_{(n+1)\delta}^\delta\big|^2=&\big|\pi^{(\delta)}(X_{n\delta}^\delta)\big|^2
+2\<\pi^{(\delta)}(X_{n\delta}^\delta),b^{(\delta)}(\pi^{(\delta)}(X_{n\delta}^\delta))\>\delta+
\big|b^{(\delta)}(\pi^{(\delta)}(X_{n\delta}^\delta))\big|^2\delta^2\\
&+\Lambda(X_{n\delta}^\delta ,\triangle W_{n\delta}),
\end{split}
\end{equation*}
where for all $x,y\in\R^d,$
\begin{align*}
\Lambda(x,y):=2\si\<\pi^{(\delta)}(x)+b^{(\delta)}(\pi^{(\delta)}(x))\delta,y\>+
\si^2|y|^2.
\end{align*}
By taking \eqref{EY*} with $r=1$ therein  and \eqref{WW*} into consideration, we derive that for any $x\in\R^d$ and $\delta\in(0,1],$
\begin{align*}
&\big|\pi^{(\delta)}(x)\big|^2
+2\<\pi^{(\delta)}(x),b^{(\delta)}(\pi^{(\delta)}(x))\>\delta+
\big|b^{(\delta)}(\pi^{(\delta)}(x))\big|^2\delta^2\\
&\le2(  C_b^\star+\psi_1(1)^2)\delta + \big(1-2 \lambda_b^\star\delta +2\psi_2(1)^2  \delta^{ 2(1- \theta)}\big)\big|\pi^{(\delta)}(x)\big|^2 \\
&\le 2(  C_b^\star+\psi_1(1)^2)\delta+ (1-  \lambda_b^\star\delta  )\big| x |^2,
\end{align*}
where  the second inequality holds true thanks to \eqref{EW3-}  and $|\pi^{(\delta)}(x)|\le|x|, x\in\R^d$. Consequently, we arrive at
\begin{equation*}
\big|X_{(n+1)\delta}^\delta\big|^2\le(1-  \lambda_b^\star\delta  )\big| X_{n\delta}^\delta \big|^2+2(  C_b^\star+\psi_1(1)^2)\delta+
\Lambda\big(X_{n\delta}^\delta ,\triangle W_{n\delta}\big).
\end{equation*}
This, along with  the binomial theorem, yields that for integer $q\ge3,$
\begin{equation}\label{HB4}
\begin{aligned}
|X^{\dd}_{(n+1)\dd}|^{2q}
 &\le (1- \lambda_b^\star\dd)^q|X^{\dd}_{n\dd}|^{2q}\\
&\quad+q(1- \lambda_b^\star\dd)^{q-1}|X^{\dd}_{n\dd}|^{2(q-1)}\big(2(  C_b^\star+\psi_1(1)^2)\delta+\Lambda(X_{n\delta}^\delta ,\triangle W_{n\delta})\big)\\
		 &\quad+ \sum_{i=0}^{q-2}C_q^i(1- \lambda_b^\star\dd)^i|X^{\dd}_{n\dd}|^{2i}\big(2(  C_b^\star+\psi_1(1)^2)\delta+\Lambda(X_{n\delta}^\delta ,\triangle W_{n\delta})\big)^{q-i}\\
&=:(1- \lambda_b^\star\dd)^q|X^{\dd}_{n\dd}|^{2q}+\Gamma_q(X_{n\delta}^\delta ,\triangle W_{n\delta})+\sum_{i=0}^{q-2}(1- \lambda_b^\star\dd)^i \Gamma_{q,i}(X_{n\delta}^\delta ,\triangle W_{n\delta}).
\end{aligned}
\end{equation}

Via the tower property of conditional expectations, together with the fact that $\triangle W_{n\delta}$  is independent of $\mathscr F_{n\delta}$,
we infer from \eqref{EW3-} that
\begin{equation}\label{T4}
\begin{split}
\E\big(\Gamma_q(X_{n\delta}^\delta ,\triangle W_{n\delta})\big|\mathscr F_0\big)&=\E\big(\E\big(\Gamma_q(X_{n\delta}^\delta ,\triangle W_{n\delta})\big|\mathscr F_{n\delta}\big)\big|\mathscr F_0\big)\\
&=q(1- \lambda_b^\star\dd)^{q-1}\big(2 (  C_b^\star+\psi_2(1)^2 )+d\sigma^2\big)\delta\E\big(|X^{\dd}_{n\dd}|^{2(q-1)}\big|\mathscr F_0\big)\\
&\le  q \big(2(  C_b^\star+\psi_2(1)^2 )+d\sigma^2\big)\delta\E\big(|X^{\dd}_{n\dd}|^{2(q-1)}\big|\mathscr F_0\big).
\end{split}
\end{equation}
Furthermore,  for an integer $k\ge1$,
\eqref{EY*} with $r=1$, besides  $|\pi^{(\delta)}(x)|\le |x|, x\in\R^d$, implies that there exists a constant $c_3(k)>0$ such that for any $\delta\in(0,1]$ and $x,y\in\R^d,$
\begin{align*}
|\Lambda(x,y)|^k&\le 3^{k-1}\big((2|\si|)^k(|x|^k+|b^{(\delta)}(\pi^{(\delta)}(x))|^k\delta^k)|y|^k+\si^{2k}|y|^{2k}\big)\\
&\le c_3(k)\big((1+|x|^k)|y|^k+|y|^{2k}\big).
\end{align*}
As a consequence, applying Young's inequality and leveraging  the fact that
\begin{align}\label{^1}
\E | W_{ \delta}|^{2k}  \le    (2k)!(\dd  d)^k/{(2^kk!)},
\end{align}
we obtain  that there exist  constants  $c_4(q),c_5(q) >0$ such that  for   $i=0,\cdots,q-2,$ and $\delta\in(0,1],$
\begin{equation}\label{T5*}
\begin{aligned}
&\E\big(\Gamma_{q,i}(X_{n\delta}^\delta ,\triangle W_{n\delta})\big|\mathscr F_0\big)\\
&\le c_4(q) \Big(\dd^{q-i}\E\big(|X^{\dd}_{n\dd}|^{2i}\big|\mathscr F_0\big)+\E\big((|X^{\dd}_{n\dd}|^{q+i}+ |X^{\dd}_{n\dd}|^{2i} )\big|\mathscr F_0\big)\E|\triangle W_{n\delta}|^{q-i}   \\
&\qquad\qquad+\E\big(|X^{\dd}_{n\dd}|^{2i}\big|\mathscr F_0\big)\E |\triangle W_{n\delta}|^{2(q-i)} \Big)\\
&\le \frac{\lambda_b^\star\dd}{4(q-1)}\E\big(|X^{\dd}_{ n \dd}|^{2q}\big|\mathscr F_0\big)+c_5(q) (1+d^q)  \delta.
\end{aligned}
\end{equation}
Whereafter, taking \eqref{T4} and \eqref{T5*} into consideration, and applying Young's inequality once more gives that for some constant $c_6(q)>0,$
\begin{align*}
\E\big(|X^{\dd}_{(n+1)\dd}|^{2q}\big|\mathscr F_0\big)\le (1-\lambda_b^\star\dd/2)\E\big(|X^{\dd}_{ n \dd}|^{2q}\big|\mathscr F_0\big)+c_6(q)\delta (1+d^{q}).
\end{align*}
So, \eqref{E8} follows by an inductive argument and the basic inequality: $a^r\le \e^{-(1-a)r},a,r>0.$
\end{proof}

In the following context, we shall write  $(X_t^\xi)_{t\ge0}$ and $(X_{n\delta}^{\delta,\xi})_{n\ge0}$  as solutions to \eqref{E1} and \eqref{E2}, respectively, with the initial value $X_0^\xi=X_0^{\delta,\xi}=\xi\in L^p(\OO\to\R^d, \mathscr F_0,\P)$ for $p\ge2,$ and,  in particular,
set the initial distribution $\mathscr L_\xi=\pi_\8,$
where $\pi_\8$ is the IPM of $(X_t)_{t\ge0.}$ Concerning the exact solution and the numerical scheme starting from the same initial distribution $\pi_\8,$
the following lemma reveals   the  corresponding convergence rate under the $L^2$-Wasserstein distance in a finite horizon.

\begin{lemma}\label{lemma2} Assume that $({\bf A}_0)$, $({\bf H}_1)$,  $({\bf H}_4)$, $({\bf H}_5)$ hold, and suppose further $({\bf A}_5)$ in case  $\beta $ given in $({\bf H}_5)$
satisfies  $\beta> 1/2.$. Then,  there exists a constant $C_0^*>0$ independent of the dimension $d$ such that
for   any $n\ge0$ and
 $\delta\in(0,\delta_{\star\star}]$,
\begin{equation}\label{LE1}
\W_2\big(\pi_\8 P^{(\dd)}_{n\dd},\pi_\8\big)^2\le C_0^*\e^{C_0^*n\delta}  \delta^{2(1\wedge \beta)}  d^{\ell_\star},
\end{equation}
where $\delta_{\star\star}:=\delta_R\wedge (2K_R^*)^{-1}\wedge (K_R^*/K_R^2)^{\frac{1}{1-2 \theta}}$ and $\ell_\star$ was defined in \eqref{DD5}.

\end{lemma}

\begin{proof}
Below, we shall stipulate $\delta\in(0,\delta_{\star\star}]$.
By invoking the invariance of $\pi_\8$,
 we have  for   $n\ge0,$
\begin{align*}
	\W_2\big(\pi_\8 P^{(\dd)}_{n\dd},\pi_\8\big)^2=\W_2\big(\pi_\8 P^{(\dd)}_{n\dd},\pi_\8P_{n\delta}\big)^2\le \E\big|Z_{n\delta}^{\delta,\xi}\big|^2,
\end{align*}
where $Z_{n\delta}^{\delta,\xi}:=X_{n\delta}^\xi-X_{n\delta}^{\delta,\xi}$. Hence,
 to achieve   \eqref{LE1}, it suffices to verify that there exists a  constant  $C_1^* >0$ such that for any $n\ge0,$
\begin{align}\label{ET-1}
\E\big|Z_{n\delta}^{\delta,\xi}\big|^2\le  C_1^*\e^{C_1^*n\delta}\delta^{2(1\wedge \beta)}d^{ \ell_\star }.
\end{align}

From \eqref{E1} and \eqref{E2}, it is easy to see that
\begin{equation}\label{^6^}
\begin{aligned}
Z_{(n+1)\delta}^{\delta,\xi}
&=X_{n\delta}^\xi-\pi^{(\delta)}(X_{n\delta}^{\delta,\xi})+\int_{n\delta}^{(n+1)\delta}\big(b(X_s^\xi)-b^{(\delta)}(\pi^{(\delta)}(X_{n\delta}^{\delta,\xi}))\, \big)\d s\\
&=\phi^{(\dd)}(X_{n\delta}^\xi)+\psi^{(\dd)}(X_{n\delta}^\xi,X_{n\delta}^{\delta,\xi}) +\int_{n\delta}^{(n+1)\delta}\big(b(X_s^\xi)-b(X_{n\delta}^\xi)\big)\,\d s,
\end{aligned}
\end{equation}
where the quantities $\phi^{(\dd)}$ and $\psi^{(\dd)}$ are defined respectively as follows:
for any $x,y\in\R^d,$
\begin{align*}
\psi^{(\dd)}(x,y):&=\pi^{(\dd)}(x)-\pi^{(\dd)}(y)+\big(b^{(\dd)}(\pi^{(\dd)}(x))- b^{(\dd)}(\pi^{(\dd)}(y))\big)\delta \\
\phi^{(\dd)}(x):&=x-\pi^{(\dd)}(x)+\dd(b(x)-b^{(\dd)}(\pi^{(\dd)}(x))).
\end{align*}

We first show  \eqref{ET-1} for the case   $\beta\in(0,1/2]$.  By means of the basic inequality:
\begin{align*}
(a+b+c)^2\le (1+2/\delta)a^2+(1+2\delta)b^2+(1+\delta+1/\delta)c^2,\quad a,b,c\in\R,
\end{align*}
along with   H\"older's inequality, we obtain that
\begin{equation}\label{KK}
\begin{split}
	\E\big|Z_{(n+1)\delta}^{\delta,\xi}\big|^2\le &  (1+ 2/\dd )\E\big|\phi^{(\dd)}(X_{n\delta}^\xi)\big|^2+(1+2\dd)\E\big|\psi^{(\dd)}(X_{n\delta}^\xi,X_{n\delta}^{\delta,\xi}) \big|^2\\
&+(1+\dd+\dd^2)\int_{n\delta}^{(n+1)\delta}\E\big|b(X_s^\xi)-b(X_{n\delta}^\xi)\big|^2\,\d s.
\end{split}
\end{equation}
By combining (${\bf A}_0$)  with (${\bf H}_5$) for $\alpha=2$ therein, there is a constant $C_2^*>0$ such that   for  $x\in\R^d$,
\begin{equation}\label{^9^}
\begin{aligned}
\big|\phi^{(\dd)}(x)\big|&\le \big|x-\pi^{(\dd)}(x)\big|+ \big|b(x)-b (\pi^{(\dd)}(x))\big|\delta +\big|b (\pi^{(\dd)}(x)) -b^{(\dd)}(\pi^{(\dd)}(x)) \big|\delta\\
&\le \big(1+L_0(1+|\pi^{(\dd)}(x)|^{\ell_0}+ |x|^{\ell_0}  )\delta\big)	\big|x-\pi^{(\dd)}(x)\big| +\lambda_{b}^* \big(1+|\pi^{(\dd)}(x)|^{l_\beta^{*}}\big)\delta^{1+\beta}\\
&\le C_2^*\big(1+|x|^{\ell_0+l_2 } + |x|^{l_\beta^{*}}\big)\delta^{2\wedge(1+\beta)}
\end{aligned}
\end{equation}
 where in the third inequality we used  $|\pi^{(\dd)}(x)|\le|x|$. Then,
by taking advantage of $\mathscr L_{X_t^\xi}=\pi_\8, ~t\ge0,$ and \eqref{EW2*},
there is a  constant $  C_3^*>0 $ such that
\begin{align}\label{^2^}
 (1+ 2/\dd )\E|\phi^{(\dd)}(X_{n\delta}^\xi)|^2\le   C_3^*\delta^{3\wedge(1+2\beta)} (1+
  d^{\ell_0+l_2}+d^{l_\beta^{*}} ).
\end{align}
Next, by virtue of \eqref{E6} and \eqref{E7}, it follows   that for all $x,y\in\R^d,$
\begin{align}\label{ER}
\big|\psi^{(\dd)}(x,y)\big| \le (1+C_R\delta)|x-y|
\end{align}
so   there exists a   constant $  C_4^*>0$ satisfying that
\begin{equation}\label{^3^}
\begin{split}
(1+2\dd)\E\big|\psi^{(\dd)}(X_{n\delta}^\xi,X_{n\delta}^{\delta,\xi}) \big|^2&\le(1+2\dd)(1+C_R\delta)^2\E\big|Z_{n\delta}^{\delta,\xi}\big|^2  \le (1+  C_4^*\delta)\E\big|Z_{n\delta}^{\delta,\xi}\big|^2.
\end{split}
\end{equation}
By invoking (${\bf A}_0$), besides H\"older's inequality, it is ready to see   that for some constant $C_5^*>0,$
\begin{equation}\label{^0^}
\begin{aligned}
 \int_{n\delta}^{(n+1)\delta}\E |b(X_s^\xi)-b(X_{n\delta}^\xi) |^2\,\d s&\le 3L_0^2 \int_{n\delta}^{(n+1)\delta}\E\big( (1+|X_s^\xi|^{2\ell_0}+|X_{n\delta}^\xi|^{2\ell_0} )  |X_s^\xi-X_{n\delta}^\xi |^2\big)\,\d s\\
 &\le 3L_0^2 \int_{n\delta}^{(n+1)\delta} \Big(1+\big(\E|X_s^\xi|^{4\ell_0}\big)^{\frac{1}{2}}+\big(\E|X_{n\delta}^\xi|^{4\ell_0}\big)^{\frac{1}{2}} \Big)\\
 &\qquad\qquad\quad\quad\quad \times\big(\E |X_s^\xi-X_{n\delta}^\xi|^4\big)^{\frac{1}{2}}\,\d s\\
 &\le  C_5^*  (1+ d^{2\ell_0+1}   )\delta^2,
\end{aligned}
\end{equation}
where in the last line we employed the fact that, for any $q\ge2$, there is a constant    $C_6^*=C_6^*(q)>0$ such that
\begin{align}\label{EE3}
\E|X_s^\xi-X_{n\delta}^\xi|^q\le C_6^*\big(1+
  d^{\frac{1}{2}q(1+\ell_0)}\big) \delta^{\frac{q}{2}}.
\end{align}
Correspondingly, we reach that
\begin{align}\label{^4^}
(1+\dd+\dd^2)\int_{n\delta}^{(n+1)\delta}\E\big|b(X_s^\xi)-b(X_{n\delta}^\xi)\big|^2\,\d s&\le 3  C_5^*\delta^2   (1+   d^{2\ell_0+1}
 ).
\end{align}
 By plugging \eqref{^2^}, \eqref{^3^} as well as  \eqref{^4^}  into \eqref{KK}, we derive  that there exists a  constant  $  C_7^* >0 $ such that,
 \begin{align}\label{^8^}
 \E\big|Z_{(n+1)\delta}^{\delta,\xi}\big|^2\le(1+  C_4^*\delta)\E\big|Z_{n\delta}^{\delta,\xi}\big|^2+   C_7^* d^{\ell_\star}\delta^{2\wedge(1+2\beta)}.
 \end{align}
 Thereafter, employing an inductive argument and using $|Z_{0}^{\delta,\xi}|=0$ as well as the fact that $1+r\le\e^r,r\ge0,$
  yields that
 \begin{equation} \label{DD2}
 \begin{split}
  \E\big|Z_{ n \delta}^{\delta,\xi}\big|^2 \le C_7^* d^{\ell_\star}\delta^{2\wedge(1+2\beta)}\sum_{i=0}^{n-1}(1+  C_4^*\delta)^i
  &\le C_7^* d^{\ell_\star}\delta^{2\wedge(1+2\beta)}\int_0^n\e^{\lfloor s\rfloor C_4^*\delta}\,\d s\\
  &\le \frac{C_7^* d^{\ell_\star}}{C_4^*}\delta^{2\beta}\e^{ C_4^*n\delta},
 \end{split}
 \end{equation}
 where   $\lfloor\cdot\rfloor$ means the floor function.

In the sequel, we aim at proving \eqref{ET-1}  for the setting $\beta>1/2.$
By making use of the basic inequality: $2ab\le \delta a^2+\delta^{-1}b^2$ for $a,b>0$, as well as H\"older's inequality once again,
we deduce that  for $\delta\in(0,1],$
\begin{equation}\label{EW1}
\begin{split}
\big|Z_{(n+1)\delta}^{\delta,\xi}\big|^2
&\le(1+2/\delta)\big|\phi^{(\dd)}(X_{n\delta}^\xi)\big|^2+(1+\delta)\big|\psi^{(\dd)}(X_{n\delta}^\xi,X_{n\delta}^{\delta,\xi}) \big|^2\\
&\quad+(\delta+\delta^2) \int_{n\delta}^{(n+1)\delta}\big|b(X_s^\xi)-b(X_{n\delta}^\xi)\big|^2\,\d s\\
&\quad+2\int_{n\delta}^{(n+1)\delta}\big\<\psi^{(\dd)}(X_{n\delta}^\xi,X_{n\delta}^{\delta,\xi}),b(X_s^\xi)-b(X_{n\delta}^\xi)\big\>\,\d s.
\end{split}
\end{equation}
Hereinafter, we intend  to quantify respectively  the four terms on the right hand side of \eqref{EW1}. Since  the first three terms have been handled respectively in \eqref{^2^}, \eqref{^3^} and   \eqref{^0^},   we  focus on the estimate on the last term in the following analysis.
  Note from the fundamental theorem of calculus that
\begin{align*}
&\int_{n\delta}^{(n+1)\delta}\<\psi^{(\dd)}(X_{n\delta}^\xi,X_{n\delta}^{\delta,\xi}),(b(X_s^\xi)-b(X_{n\delta}^\xi))\>\,\d s\\
&=\int_{n\delta}^{(n+1)\delta}\int_0^1\frac{\d}{\d u}\<\psi^{(\dd)}(X_{n\delta}^\xi,X_{n\delta}^{\delta,\xi}), b(X_{n\delta}^\xi+u(X_s^\xi-X_{n\delta}^\xi))\>\,\d u\,\d s\\
&=\int_{n\delta}^{(n+1)\delta}\int_0^1 \<\psi^{(\dd)}(X_{n\delta}^\xi,X_{n\delta}^{\delta,\xi}), \nn b(X_{n\delta}^\xi+u(X_s^\xi-X_{n\delta}^\xi))\cdot(X_s^\xi-X_{n\delta}^\xi))\>\,\d u\,\d s\\
&= \int_{n\delta}^{(n+1)\delta} \big\<\psi^{(\dd)}(X_{n\delta}^\xi,X_{n\delta}^{\delta,\xi}),\nn b(X_{n\delta}^\xi)\cdot(X_s^\xi-X_{n\delta}^\xi)   \big\>\,\d s\\
&\quad+\int_{n\delta}^{(n+1)\delta}\int_0^1\big\<\psi^{(\dd)}(X_{n\delta}^\xi,X_{n\delta}^{\delta,\xi}),\big(\nn b(X_{n\delta}^\xi+u (X_s^\xi-X_{n\delta}^\xi))-\nn b(X_{n\delta}^\xi )\big)  (X_s^\xi-X_{n\delta}^\xi)\big\>\,\d u\,\d s\\
&=:\Pi_1^{(\delta)}(n\delta)+\Pi_2^{(\delta)}(n\delta),
\end{align*}
where the second identity is valid by means of the chain rule, and
the third identity is available thanks to the addition-subtraction technique.

Below, we intend to estimate the terms $\Pi_1^{(\delta)}$ and $\Pi_2^{(\delta)}$, one by one. On the one hand, by taking advantage of  $({\bf A}_5)$, \eqref{ER} and
\begin{align*}
X_s^\xi-X_{n\delta}^\xi=\int_{n\delta}^sb(X_u^\xi)\,\d u+W_s-W_{n\delta},\quad s\in[n\delta,(n+1)\delta],
\end{align*}
 we infer from \eqref{EW2*}  that there exist some constants $C_1^\star,C_2^\star>0$ such that for any $\delta\in(0,1],$
\begin{align*}
\E\Pi_1^{(\delta)}(n\delta)&=\int_{n\delta}^{(n+1)\delta}\int_{n\delta}^s\E \big\<\psi^{(\dd)}(X_{n\delta}^\xi,X_{n\delta}^{\delta,\xi}),\nn b(X_{n\delta}^\xi)\cdot b(X_u^\xi) \big\>\,\d u \,\d s\\
&\le \int_{n\delta}^{(n+1)\delta}\int_{n\delta}^s\E \big|\psi^{(\dd)}(X_{n\delta}^\xi,X_{n\delta}^{\delta,\xi})\big|\cdot\big|\nn b(X_{n\delta}^\xi)\cdot b(X_u^\xi) \big|\,\d u \,\d s\\
&\le C_1^\star\int_{n\delta}^{(n+1)\delta}\int_{n\delta}^s\E\Big(  \big| Z_{n\delta}^{\delta,\xi} \big|   \big(1+|X_{n\delta}^\xi|^{\ell_0}\big) \big(1+|X_u^\xi|^{1+\ell_0}\big)\Big)\,\d u \,\d s\\
&\le \frac{1}{2} C_1^\star\int_{n\delta}^{(n+1)\delta}\int_{n\delta}^s\E\Big( \delta^{-1}\big| Z_{n\delta}^{\delta,\xi} \big|^2 +\delta \big(1+|X_{n\delta}^\xi|^{\ell_0}\big)^2 \big(1+|X_u^\xi|^{1+\ell_0}\big)^2\Big)\,\d u \,\d s\\
&\le \frac{1}{2}C_1^\star\delta\E\big|Z_{n\delta}^{\delta,\xi}\big|^2+C_2^\star\big(1+
  d^{2\ell_0+1}\big)\delta^3,
\end{align*}
where in the identity we exploited the fact that the increment $W_s-W_{n\delta}$ is independent of $\mathscr F_{n\delta}$ and  $\E(W_s-W_{n\delta})={\bf0}$. On the other hand,   by means of (${\bf A}_5$), \eqref{ER} as well as \eqref{EE3}, there exist constants $C_3^\star,C_{4}^\star>0$ such that for any $\delta\in(0,1],$
\begin{align*}
\E\Pi_2^{(\delta)}(n\delta)&\le	 C_3^\star\int_{n\delta}^{(n+1)\delta}  \E\Big( \big| Z_{n\delta}^{\delta,\xi} \big| \big(1+\big|X_{n\delta}^\xi\big|^{\ell^\star_0}+\big|X_s^\xi\big|^{\ell^\star_0}\big)   \big|X_s^\xi-X_{n\delta}^\xi\big|^2\Big)  \,\d s\\
&\le \frac{1}{2}C_3^\star\int_{n\delta}^{(n+1)\delta}  \E\Big( \big| Z_{n\delta}^{\delta,\xi} \big|^2+ \big(1+\big|X_{n\delta}^\xi\big|^{\ell^\star_0}+\big|X_s^\xi\big|^{\ell^\star_0}\big)^2   \big|X_s^\xi-X_{n\delta}^\xi\big|^4\Big)  \,\d s\\
&\le \frac{1}{2}C_3^\star\delta\E\big| Z_{n\delta}^{\delta,\xi} \big|^2+C_4^\star\big(
 1+d^{2(\ell_0+1)+\ell_0^\star}  \big)\delta^3.
\end{align*}
Therefore, we deduce that
\begin{equation}\label{^5}
\begin{aligned}
&\int_{n\delta}^{(n+1)\delta}\E\<\psi^{(\dd)}(X_{n\delta}^\xi,X_{n\delta}^{\delta,\xi}),(b(X_s^\xi)-b(X_{n\delta}^\xi))\>\,\d s\\
& \le\frac{1}{2}\big(C_1^\star\vee C_3^\star\big)\delta\E\big| Z_{n\delta}^{\delta,\xi} \big|^2 +2\big(C_2^\star\vee C_{4}^\star\big)\big(1+
 d^{2(\ell_0+1)+\ell_0^\star}  \big)\delta^3.
\end{aligned}
\end{equation}
Now, substituting \eqref{^2^}, \eqref{^3^}, \eqref{^0^} as well as \eqref{^5} into \eqref{EW1} enables us to show  that there are constants $C_{5}^\star,C_{6}^\star>0 $ such that for all $\delta\in(0,1],$
 \begin{align*}
 \E\big|Z_{(n+1)\delta}^{\delta,\xi}\big|^2&\le(1+C_{5}^\star\delta)\E\big|Z_{n\delta}^{\delta,\xi}\big|^2+  C_{6}^\star\delta^{3\wedge(1+2\beta)} d^{\ell_\star}.
 \end{align*}
Whence, \eqref{ET-1} follows for the case $\beta>1/2$ by repeating the procedure to derive \eqref{DD2}.
\end{proof}

With Lemma  \ref{lemma1},  Lemma \ref{lemma3} as well as Lemma \ref{lemma2}  at hand, we aim to complete the proof of Theorem \ref{IPM}.

\begin{proof}[Proof of Theorem \ref{IPM}]
 By virtue of $({\bf H}_4)$,
 $(X_t)_{t\ge0}$ admits an IPM  $\pi_\8,$  which is $L^p$-integrable for any $p>0$; see Lemma \ref{lemma1} for more details. In the sequel, we stipulate $\delta\in(0,\delta_2^{\star\star}]$.
 Let  $\pi_\8$ and ${\tt \pi}_\8$ be two IPMs of $(X_t)_{t\ge0}$. By the triangle inequality,   one has
 \begin{align*}
 	\W_2(\pi_\8,{\tt \pi}_\8)\le \W_2(\pi_\8,{\pi}^{(\dd)}_\8)+\W_2(\pi^{(\dd)}_\8,{\tt \pi}_\8).
 \end{align*}
Whence, the uniqueness of IPMs for  $(X_t)_{t\ge0}$ follows once the assertion \eqref{E6-1} has been claimed.

   Obviously, with the help of  \eqref{EW2-}, it follows that for all $\delta\in(0,1],$
      $$ |b^{(\delta)}({\bf0})|< \big|b^{(\dd)}({\bf 0})-b({\bf 0})\big|+|b({\bf 0})|\le \ll_b^*+|b({\bf 0})|.
      $$
Therefore,   we have for all $x\in\R^d,$
\begin{align*}
 \<\pi^{(\delta)}(x),b^{(\delta)}(\pi^{(\delta)}(x))\>&\le\<\pi^{(\delta)}(x),b^{(\delta)}(\pi^{(\delta)}(x))-b^{(\delta)}({\bf0})\>+\big(\ll_b^*+|b({\bf 0})|\big)|\pi^{(\delta)}(x)|.
 \end{align*}
 This, along with \eqref{E*}, yields that for any $x\in \R^d$ with $|x|>R,$
 \begin{align*}
 \<\pi^{(\delta)}(x),b^{(\delta)}(\pi^{(\delta)}(x))\>&\le-\frac{1}{2}K_R^*|\pi^{(\delta)}(x)|^2+\frac{1}{2K_R^*}\big(\ll_b^*+|b({\bf 0})|\big)^2.
 \end{align*}
Additionally, by means of $|\pi^{(\delta)}(x)|\le|x|$, $({\bf A}_0)$, as well as \eqref{EW2-}, we obtain that for any $\delta\in(0,1]$ and $x\in\R^d$ with $|x|\le R,$
\begin{align*}
 \<\pi^{(\delta)}(x),b^{(\delta)}(\pi^{(\delta)}(x))\>&\le R|b^{(\delta)}(\pi^{(\delta)}(x))-b (\pi^{(\delta)}(x))|+R|b (\pi^{(\delta)}(x))-b({\bf0})|+R|b({\bf 0})|\\
 &\le R \ll_b^* (1+R^{l_\beta^{*}}) +R\big(L_0(1+R^{\ell_0})R+|b({\bf 0})|\big).
 \end{align*}
Therefore,  the hypothesis  \eqref{WW*} is valid for  $ \ll_b^\star=\ff12K_R^*$ so  Lemma \ref{lemma3} yields that $(X_{n\delta}^{\delta})_{n\ge1}$ has an IPM, denoted by $\pi_\8^{(\delta)}$, and
Theorem \ref{thm} further implies that $\pi_\8^{(\delta)}$ is in fact the unique IPM of $(X_{n\delta}^{\delta})_{n\ge1}$.

 By invoking the invariance of $\pi_\8$ and $\pi_\8^{(\delta)}$, it follows readily from Theorem \ref{thm} that for any   $n\ge1$ and $\dd\in(0,\dd_2^{\star\star}],$
 \begin{equation}\label{EW3*}
 \begin{split}
	\W_2\big(\pi_\8,\pi_\8^{(\delta)}\big)&=\W_2\big(\pi_\8P_{n\delta},\pi_\8^{(\delta)}P_{n\delta}^{(\delta)}\big)\\
&\le \W_2\big(\pi_\8 P^{(\dd)}_{n\dd},\pi_\8^{(\delta)} P^{(\dd)}_{n\dd}\big)+\W_2\big(\pi_\8 P_{n\dd},\pi_\8 P^{(\dd)}_{n\dd}\big)\\
&\le C_0\e^{-\ff{1}2(\lambda_1\wedge \lambda_2) n\delta}\W_2\big(\pi_\8,\pi_\8^{(\delta)}\big)+\W_2\big(\pi_\8 P_{n\dd},\pi_\8 P^{(\dd)}_{n\dd}\big),
\end{split}
\end{equation}
where the constants $C_0,\lambda_1,\lambda_2$ were defined in  \eqref{WQ*}.
Now, we choose  the integer
$$n_\delta=\lfloor 2\ln (1+2C_0)/((\lambda_1\wedge \lambda_2)\delta)\rfloor+1$$  such that $C_0\e^{-\ff{1}2(\lambda_1\wedge \lambda_2) n_\delta\delta}\le\frac{1}{2}$. As a result, by taking $n=n_\delta$ in \eqref{EW3*},
 we deduce from Lemma \ref{lemma2} that
\begin{align*}
	\W_2\big(\pi_\8,\pi_\8^{(\delta)}\big)
&\le \frac{1}{2}\W_2\big(\pi_\8,\pi_\8^{(\delta)}\big)+\W_2\big(\pi_\8 P_{n_\delta\dd},\pi_\8 P^{(\dd)}_{n_\delta\dd}\big)\\
 &\le \frac{1}{2}\W_2\big(\pi_\8,\pi_\8^{(\delta)}\big)+(C_0^*)^{\frac{1}{2}}\e^{\frac{1}{2}C_0^*n_\delta\delta}  \delta^{1\wedge \beta }d^{\ff{\ell_\star}2}
\end{align*}
so that
\begin{align}\label{DD3}
\W_2\big(\pi_\8,\pi_\8^{(\delta)}\big)&\le 2(C_0^*)^{\frac{1}{2}}\e^{\frac{1}{2}C_0^*n_\delta\delta}  \delta^{1\wedge \beta }d^{\ff{\ell_\star}2}\le 2(C_0^*)^{\frac{1}{2}}\e^{\ff12 C_0^*}(1+2C_0)^{\ff{	C_0^*}{\ll_1\wedge \ll_2}}  \delta^{1\wedge \beta }d^{\ff{\ell_\star}2},
\end{align}
where the second inequality holds true by invoking
$$n_\delta\delta\le   \frac{2\ln (1+2C_0)}{\lambda_1\wedge \lambda_2}  +1.$$

Next, for notation brevity, we set
\begin{align*}
C_*&:= 8(1+ R) \big( (3C_R(K_R^*+96 C_R)(1+R))   \vee     (\psi_1( R)K_R^*)  \vee(12C_R  \psi_1(R_0) )  \big)/{K_R^*},\\
C_{**}&:=24C_R(K_R^*+96 C_R)(1+R)^2/{K_R^*}.
\end{align*}
By virtue of  (${\bf H}_3$), we   obviously have $d\si^2>C_*.$ Then,
according to the definition of $\lambda_1$ given in \eqref{WQ*}, we find that
\begin{align*}
\lambda_1   =  3C_R\Big(\frac{2 }{1+ C_{**}/{(d\si^2)}} -1\Big)&> \frac{3C_R(C_*-C_{**}) }{C_* +C_{**} }>0,
\end{align*}
where the last inequality is valid by noting  $C_*\ge C_{**}$. As a consequence, by recalling that $\lambda_2$ is a positive constant which is unrelated to the dimension $d,$
the quantity $\e^{\frac{1}{2}C_0^*n_\delta\delta}$ can be dominated by a constant independent of the dimension $d$. Finally,
the  assertion \eqref{E6-1} is available by taking \eqref{DD3} into consideration.

Via the triangle inequality and the invariance of $\pi_\8^{(\delta)}$, besides  Theorem \ref{thm}, it is easy to see  that
\begin{align*}
\W_2(\mu P^{(\dd)}_{n\dd},\pi_\8)&\le \W_2(\mu P^{(\dd)}_{n\dd},\pi_\8P^{(\dd)}_{n\dd})+\W_2(\pi_\8 P^{(\dd)}_{n\dd},\pi_\8^{(\delta)}P^{(\dd)}_{n\dd})+\W_2(\pi_\8^{(\delta)},\pi_\8)\\
&\le
 C_0\e^{-\ff{1}2(\lambda_1\wedge \ll_2) n\delta}\mathbb W_2(\mu,\pi_\8)+(C_0+1)\W_2(\pi_\8^{(\delta)},\pi_\8).
\end{align*}
Accordingly, \eqref{E6-2} is reachable  by taking \eqref{E6-1} into consideration.
\end{proof}

As a direct application of  Theorem \ref{IPM}, the proof of Theorem \ref{cor} can be finished.

\begin{proof}[ Proof of Theorem \ref{cor}]
  In  order to complete the proof of Theorem \ref{cor}, it suffices to check all the prerequisites imposed in Theorem \ref{IPM}, one by one.

Under $({\bf A}_0)$  and (${\bf A}_4$), $(X_t)_{t\ge0}$ has an IPM by applying Lemma \ref{lemma1}. Furthermore, $({\bf A}_0)$ with $\ell_0=0$ and (${\bf A}_4$) imply that Wang's Harnack inequality holds true so $(X_t)_{t\ge0}$ has a unique IPM; see, for example, \cite[Theorem 1.4.1]{Wang}. Similarly,  $(X_t)_{t\ge0}$ also possess  a unique IPM under $({\bf A}_0)$ and \eqref{EE11-}.

Regarding the EM scheme \eqref{WW},  under $({\bf A}_0)$ with $\ell_0=0$ and (${\bf A}_4$), it is easy to see that (${\bf H}_1$) and (${\bf H}_2$)
are satisfied with $\theta=0,$ $\delta_r^*=1$, and
\begin{align*}
R=R_*,\quad  C_R=K_R=3L_0,\quad K_R^*=L_5,\quad \psi_1(r)=|b({\bf0})|+2L_0r,\quad \psi_2(r)\equiv2L_0.
\end{align*}
Furthermore, (${\bf H}_4$)   holds true with
\begin{align*}
\lambda_b= L_5/2,~C_b=L_5R_*^2+R_*\big(2L_0R_*+| b({\bf 0})|\big)
+ | b({\bf 0})|^2/{(2L_5)}
\end{align*}
and  (${\bf H}_5$) with $\beta=1$ and $l_\alpha=0$ is satisfied automatically due to $\pi^{(\delta)}={\rm Id}$ and $b^{(\delta)}=b.$ Consequently, as for the EM scheme \eqref{WW},  the proof of Theorem \ref{cor} is done by applying Theorem \eqref{IPM} with $\ell_0=\ell_0^\star=l_\alpha=l_\beta^*=0$.

Since Assumptions $({\bf A}_1)$ and $({\bf H}_1)$-$({\bf H}_3)$ have been examined in the proof of Theorem \ref{thm-0}, it is sufficient to check Assumptions $({\bf H}_4)$ and $({\bf H}_5)$ for the TEM scheme \eqref{EW0-} and the PEM scheme \eqref{EWW}, respectively. By virtue of Lemma \ref{lem3}, Assumption $({\bf H}_5)$ has been verified for both approximation schemes.
So, it remains to justify Assumption $({\bf H}_4)$, which can be guaranteed readily  under $({\bf A}_0)$ and \eqref{EE11-} or $({\bf A}_4)$.

On the basis of the preceding  analysis, the proof of Theorem \ref{cor} is complete.
\end{proof}

\section{Proof of Theorem \ref{cor1-1}}\label{sec4}
Throughout this section,   let $(X_t^\delta)_{t\ge0}$ be the corresponding  continuous-time version of the TEM scheme \eqref{RT3}, which is   given as below: for $t_\delta:=\lfloor t/\delta\rfloor \delta$,
\begin{align}\label{WQ}
		\d X^{\dd}_t=\bar b^{(\dd)}(X^{\dd}_{t_\dd})\,\d t+\si \,\d W_t.
	\end{align}
Before the accomplishment on the proof of Theorem \ref{cor1-1}, we prepare the following crucial lemma, which reveals that the convergence rate between the exact solution and the numerical one determined by \eqref{WQ} is $1$ in an interval with finite length.

\begin{lemma}\label{lemma4}
Assume that   $({\bf A}_0)$, $({\bf A}_5)$ and \eqref{EE11-} hold.
Then, there is   a constant   $  c^\star >0$
 such that for any $t\ge s\ge0$ with $|t-s|\le1$, and $\delta\in(0,\delta^\star_3],$
\begin{equation}\label{EE4}
	\E	|X_t-X_{t}^{\dd}|^2\le  c^\star\Big( \big(
 d^{\ell_{\star\star}}
	+\E|X_0|^{4(1+\ell_0)}+\E|X_0^\delta|^{p_\star}\big)\delta^2+\E|X_s-X_s^{\dd}|^2\Big)
\end{equation}
 provided $X_0\in L^{4(1+\ell_0)}(\OO\to\R^d,\mathscr F_0,\P)$
and $X_0^\delta\in L^{p_\star}(\OO\to\R^d,\mathscr F_0,\P)$, where positive constants $\delta_3^\star, p_\star, \ell_{\star\star}$ were defined in \eqref{DD6}.
\end{lemma}

\begin{proof}
Below,  we stipulate $t\ge s\ge0$ and  set   $Z_t:=X_t-X_t^\delta, t\ge0.$ Obviously, $(Z_t)_{t\ge0}$
solves the random ODE:
\begin{align*}
\d Z_t=\big(b(X_t)-\bar b^{(\dd)}(X^{\dd}_{t_\dd})\big)\,\d t.
\end{align*}
Then, the chain rule, along with the addition-subtraction trick,  enables us to derive that
\begin{align*}
|Z_t|^2
&=|Z_s|^2+2\int_s^t\<Z_u,b(X_u)-b(X^{\dd}_u)\>\,\d u +2\int_s^t\<Z_u,  b (X^{\dd}_{u_\dd})-\bar b^{(\dd)}(X^{\dd}_{u_\dd})\>\,\d u\\
&\quad+2\int_s^t\<Z_u, b(X^{\dd}_u)-b (X^{\dd}_{u_\dd})-\nn b(X^{\dd}_{u_\dd})(X_u^\delta-X^{\dd}_{u_\dd})\>\,\d u\\
&\quad+2\int_s^t\<Z_u, \nn b(X^{\dd}_{u_\dd})(X_u^\delta-X^{\dd}_{u_\dd})\>\,\d u\\
&=:|Z_s|^2+\sum_{i=1}^4\Gamma_i(s,t,\delta)\\
&=|Z_s|^2+\sum_{i=1}^3\Gamma_i(s,t,\delta) +2\int_s^t\<Z_u, \nn b(X^{\dd}_{u_\dd})\bar b^{(\delta)}(X_{u_\delta}^\delta)\>(u-u_\delta) \,\d u \\
&\quad+2\si\int_s^t\<Z_{u_\delta}, \nn b(X^{\dd}_{u_\dd})(W_u-W_{u_\delta}) \>\,\d u\\
&\quad+2\si\int_s^t\int_{u_\delta}^u\<b(X_r)-b(X_{u_\delta}),\nn b(X^{\dd}_{u_\dd})(W_u-W_{u_\delta})\> \,\d r\,\d u\\
&\quad+2\si\int_s^t\int_{u_\delta}^u\<b(X_{u_\delta})-\bar b^{(\delta)}(X_{u_\delta}^\delta),\nn b(X^{\dd}_{u_\dd})(W_u-W_{u_\delta})\> \,\d r\,\d u\\
&=:|Z_s|^2+\sum_{i=1}^3\Gamma_i(s,t,\delta)+\sum_{i=1}^4\Gamma_{4i}(s,t,\delta),
\end{align*}
where in the second inequality we explored
\begin{align*}
Z_u=Z_{u_\delta}+\int_{u_\delta}^u\big(b(X_r)-\bar b^{(\delta)}(X_{u_\delta}^\delta)\big)\,\d r~\mbox{ and } ~X_u^\delta-X^{\dd}_{u_\dd}=\bar b^{(\delta)}(X_{u_\delta}^\delta)(u-u_\delta)+\si(W_u-W_{u_\delta}).
\end{align*}
By exploiting  $({\bf A}_0)$ and \eqref{EE11-}, it follows readily that for any $x,y\in\R^d,$
\begin{align*}
\<x-y,b(x)-b(y)\>\le L_{0}(1+2(R^*)^{\ell_0})|x-y|^2.
\end{align*}
This further implies that
\begin{align*}
 \E\Gamma_{1}(s,t,\delta)\le  L_{0}(1+2(R^*)^{\ell_0})\int_s^t \E|Z_u|^2\,\d u.
\end{align*}

To quantify  the remaining terms, it is necessary  to show that $(X_{n\delta}^\delta)_{n\ge0}$, determined by \eqref{RT3}, possesses a uniform moment estimate. Note from \eqref{RT3} that
\begin{align*}
|X^\dd_{(n+1)\delta}|^2
&=|X^\dd_{n\delta}|^2+2\<X^\dd_{n\delta},\bar b^{(\delta)}(X^\dd_{n\delta})\>\delta+|\bar b^{(\delta)}(X^\dd_{n\delta})|^2\delta^2\\
&\quad+2\si\<X^\dd_{n\delta}+\bar b^{(\delta)}(X^\dd_{n\delta})\delta,\triangle W_{n\delta}\>+\si^2|\triangle W_{n\delta}|^2.
\end{align*}
By means of \eqref{EE11-}, there exists a constant $c_1>0$ such that for any $x\in\R^d $ and $\delta\in(0,1],$
 \begin{equation}\label{DD7}
 \begin{split}
 |x+\bar b^{(\delta)}(x)\delta|^2&=|x|^2+2\<x,\bar b^{(\delta)}(x)\>\delta+|\bar b^{(\delta)}(x)|^2\delta^2\\
 &\le |x|^2-\frac{ L_2(1+|x|^{\ell_0})|x|^2\delta}{(1+\delta|x|^{2\ell_0})^{\frac{1}{2}}}+2|b({\bf0})|\cdot|x|\delta+c_1\delta
 \\
 &\quad-\big( L_2(1+\delta|x|^{2\ell_0})^{\frac{1}{2}}-2L_0^2(1+|x|^{\ell_0})\delta\big)\times\frac{(1+|x|^{\ell_0})\delta|x|^2}{ 1+\delta|x|^{2\ell_0} }.
 \end{split}
 \end{equation}
 Furthermore, via  $\ss{a+b}\ge(\ss a+\ss b)/{\ss2},a,b\ge0,$ it follows that for any  $\delta\in(0,L^2_2/(32L_0^4)]$ and
 $x\in\R^d$ with $|x|\ge1$,
 \begin{align*}
 L_2(1+\delta|x|^{2\ell_0})^{\frac{1}{2}}-2L_0^2(1+|x|^{\ell_0})\delta&\ge \frac{L_2}{\ss2}\big(1+\ss\delta|x|^{\ell_0}\big)-2L_0^2(1+|x|^{\ell_0})\delta\\
 &\ge \frac{L_2}{\ss2}+\Big(\frac{L_2|x|^{\ell_0}}{\ss2(1+|x|^{\ell_0})}  -2L_0^2\ss\delta\Big)(1+|x|^{\ell_0})\ss\delta\\
 &\ge\frac{L_2}{\ss2}+ \Big(\frac{L_2 }{2\ss2 }  -2L_0^2\ss\delta\Big)(1+|x|^{\ell_0})\ss\delta>0,
 \end{align*}
 in which in the third inequality we utilized the fact that $[0,\8)\ni r\mapsto r/(1+r)$ is increasing.
 Thus, we deduce from \eqref{DD7} that for some constants $c_2,c_3>0$ such that for any $x\in\R^d$ and $\delta\in(0,\delta_3^\star],$
 \begin{equation*}
 \begin{split}
 |x+\bar b^{(\delta)}(x)\delta|^2
 &\le |x|^2-\frac{ L_2 |x|^{\ell_0+2}  \delta}{(1+\delta|x|^{2\ell_0})^{\frac{1}{2}}}+|b({\bf0})|\cdot|x|\delta+c_2\delta\\
 &\le \big(1-  L_2   \delta/(2\ss2)\big)|x|^2  +c_3\delta.
 \end{split}
 \end{equation*}
 where in the second inequality we exploited the fact that $[0,\8)\ni r\mapsto r/(1+\delta r^2)^{\frac{1}{2}}$ is non-decreasing. Subsequently, we deduce that
 \begin{align*}
|X_{(n+1)\delta}^\delta|^2
&\le\big(1-  L_2   \delta/(2\ss2)\big)|X_{n\delta}^\delta|^2+c_3\delta+2\si\<X_{n\delta}^\delta+b^{(\delta)}(X_{n\delta}^\delta)\delta,\triangle W_{n\delta}\>+\si^2|\triangle W_{n\delta}|^2.
\end{align*}
Whence,  by following the proof of   Lemma \ref{lemma3}, for $X_0^\delta\in L^{p}(\OO\to\R^d,\mathscr F_0,\P), p>0,$ there exists a constant  $c_4=c_4(p)>0$   such that for any $\delta\in(0,\delta^\star_3] $ and $n\ge0,$
\begin{align}\label{EWQ}
 \E|X_{n\dd}^\dd|^p\le c_4\big(1+d^{\ff{p}2}\big)+ \E|X_0^\delta|^{p}.
\end{align}
From \eqref{WQ}, we obviously have
\begin{align*}
X_t^\dd-X_{t_\dd}^\dd=\bar b^{(\dd)}(X_{t_\dd}^\dd)(t-t_\delta)+W_t-W_{t_\dd}.
\end{align*}
Subsequently,
by employing
the $C_r$-inequality, \eqref{^1}    as well as
the stationary increment of $(W_t)_{t\ge0}$, for  $X_0^\delta\in L^{p }(\OO\to\R^d,\mathscr F_0,\P), p>0$,  we deduce from \eqref{EWQ} that there exist  constants $c_5=c_5(p),c_6=c_6(p)>0$ such that
\begin{equation}\label{EE4-1}
\begin{aligned}
\E|X_{t_\dd}^\dd-X_t^\dd|^{p} &\le  2^{(p-1)^+} \bigg(\delta^{p}\E|\bar b^{(\dd)}(X_{t_\dd}^\dd)|^{p}+\ff{1}{2^{p/2}}\Big(\ff{(2\lceil p\rceil)!}{\lceil p\rceil !}\Big)^{\frac{p}{2\lceil p\rceil}}(d\delta)^{\ff{p}{2}}\bigg)\\
 &\le c_5\dd^{ \ff{p}2}\big(1+d^{\frac{ p}{2}}+\E|X_{t_\dd}^\dd|^{p}\big)\\
 &\le c_6\dd^{ \ff{p}2}\big(1+d^{\frac{ p}{2}}+\E|X_0^\dd|^{p}\big),
\end{aligned}
\end{equation}
where in the second inequality we employed that
\begin{align*}
|\bar b^{(\delta)}(x)|\le |b({\bf0})|+\ss 2L_0\delta^{-\frac{1}{2}}|x|,\quad x\in\R^d.
\end{align*}
Notice from $({\bf A}_5)$  that there exists a constant $c_7>0$ such that
 \begin{align*}
 |b(x)-b(y)-\nn b(x)(x-y)|\le c_7\big(1+|x|^{\ell_0^\star}+|y|^{\ell_0^\star}\big)|x-y|^2,\quad x,y\in\R^d.
\end{align*}
Whereafter,
 we  deduce from   \eqref{EE4-1} and  Lemma \ref{lem3}  that for some constant $c_8>0,$
\begin{align*}
 \E\Gamma_{2}(s,t,\delta)+\E  \Gamma_{3}(s,t,\delta)+ &\E\Gamma_{41}(s,t,\delta)\le 3\int_s^t\E|Z_u|^2\,\d u+  c_8\big(
d^{\ell_{\star\star}}
 + \E|X_0^\delta|^{p_\star}\big)(t-s)\delta^2,
\end{align*}
 where in the inequality we also utilized that for some constant $c_9>0,$
 \begin{align*}
|\nn b(x)\bar b^{(\delta)}(x)|\le c_9\big(1+|x|^{1+ 2\ell_0}\big),\quad x\in\R^d,
\end{align*}
by taking $({\bf A}_0)$ into account.
Concerning  $\Gamma_{42}$ and $\Gamma_{44}$, it is trivial to see that $\E\Gamma_{42}(s,t,\delta)=\E\Gamma_{44}(s,t,\delta)=0$ due to  the fact that $W_u-W_{u_\delta}$ is independent of $\mathscr F_{u_\delta}$.
Finally, notice that $({\bf H}_4)$ is satisfied by \eqref{EE11-} and $({\bf A}_0)$. Therefore, via  H\"older's inequality, along with $X_0\in L^{4(\ell_0+1)}(\OO\to\R^d,\mathscr F_0,\P)$ and $X_0^\delta\in L^{p_\star}(\OO\to\R^d,\mathscr F_0,\P)$,
 we obtain from \eqref{WW2}   and \eqref{EWQ} that for some constant $c_{10}>0,$
 \begin{align*}
\E\Gamma_{43}(s,t,\delta)
&\le2|\si|\int_s^t\int_{u_\delta}^u\big(\E|b(X_r)-b(X_{u_\delta})|^2\big)^{\frac{1}{2}}\big(\E\|\nn b(X^{\dd}_{u_\dd})\|_{\rm op}^2\big)^{\frac{1}{2}}\big(\E|W_{u-u_\delta}|^2\big)^{\frac{1}{2}} \,\d r\,\d u\\
&\le c_{10} \big(1+{ d^{3\ell_0/2+1} }
+\E|X_0|^{4(1+\ell_0)}+\E|X_0^\delta|^{p_\star}\big)\delta^2(t-s),
\end{align*}
 where in the second inequality we also made use of $({\bf A}_0)$ and fact
 for any $q\ge2$, there is a constant  $c_{11}=c_{11}(q)>0$ such that
\begin{align*}
\E|X_s-X_{n\delta}|^q\le c_{11}\big(1+
  d^{\frac{1}{2}q(1+\ell_0)}+\E|X_0|^{q(\ell_0+1)}\big) \delta^{\frac{q}{2}}.
\end{align*}

At length,  by summing up the analysis above, we conclude that for some constant $c_{12}>0,$
\begin{align*}
\E|Z_t|^2\le \E|Z_s|^2+c_{12}\big(
d^{\ell_{\star\star}}
 +\E|X_0|^{4(1+\ell_0)}+\E|X_0^\delta|^{p_\star}\big)\delta^2(t-s)+c_{12}\int_s^t \E|Z_u|^2\,\d u .
\end{align*}
Whence, the assertion \eqref{EE4} follows from Gronwall's inequality and $|t-s|\le1$.
\end{proof}

Based on Lemma \ref{lemma4}, the proof of Theorem \ref{cor1-1} can be manipulated.
\begin{proof}[Proof of Theorem \ref{cor1-1}]
Obviously, \eqref{WE} is available once \eqref{EW**} is established and subsequently we take $\mu=\pi_\8^{(\delta)}$ in \eqref{EW**} followed by  sending $n\to\8$.

Under $({\bf A}_0)$ and \eqref{EE11-},  \cite[Theorem 1]{PM} implies that  there exist constants $C_1,\lambda_*,\si_0>0$ such that for any $\mu,\nu\in\mathscr P_2(\R^d)$, $t\ge0,$ and the noise intensity $\si$ satisfying $|\si|\ge\si_0,$
\begin{align}\label{EE5}
\mathbb W_2(\mu P_t,\nu P_t)\le C_1\e^{-\lambda_*t}\mathbb W_2(\mu  ,\nu  ),
\end{align}
where  $ \mu P_t $ stands for  the law of $ X_t $ with $\mathscr L_{X_0}=\mu.$
By the triangle inequality and the invariance of $\pi_\8$,
we obtain readily  that  for any   $n\ge0,$ and $\mu\in \mathscr P_{p^\star}(\R^d),$
\begin{align*}
	\W_2\big(\mu P^{(\dd)}_{n\dd},\pi_\8\big)\le &\W_2\big(\mu P^{(\dd)}_{n\dd},\mu  P_{n\dd}\big)+\W_2\big(\mu P_{n\dd},\pi_\8 P_{n\dd}\big).
	\end{align*}
 This, together with \eqref{EE5},  yields the desired assertion \eqref{EW**}
as long as  there exists a   constant  $C_2>0$  such that for any $n\ge0$, $\delta\in(0,\delta_3^\star],$ and $\mu\in \mathscr P_{p_\star}(\R^d),$
\begin{align}\label{RT}
\W_2\big(\mu P^{(\dd)}_{n\dd},\mu  P_{n\dd}\big)\le C_2\big(d^{\ell_{\star\star}}+\mu(|\cdot|^{p_\star})\big)\delta.
\end{align}

Below, we shall stipulate $\delta\in(0,\delta_3^\star]$.
Once $m:=\lfloor n/\lfloor \dd^{-1}\rfloor\rfloor=0$, we obviously have $n\dd<1$. For this case,
 \eqref{RT} follows  by applying Lemma \ref{lemma4} with the same initial value $X_0=X_0^\delta$. In the following analysis, we focus on the case $m\ge1.$
By invoking  Lemma \ref{lemma4} (with the same initial value at the time $m\lfloor \dd^{-1}\rfloor\dd$) and
 the triangle inequality, we infer that for some constant $C_3>0,$
\begin{equation}\label{RT1}
	\begin{aligned}
	\W_2\big(\mu P^{(\dd)}_{n\dd},\mu  P_{n\dd}\big)\le & \W_2\big(\mu P^{(\dd)}_{n\dd}, \big(\mu P^{(\dd)}_{m\lfloor \dd^{-1}\rfloor\dd}\big)P_{(n-m\lfloor \dd^{-1}\rfloor)\dd}\big)\\
&+\W_2\big(\big(\mu P^{(\dd)}_{m\lfloor \dd^{-1}\rfloor\dd}\big)P_{(n-m\lfloor \dd^{-1}\rfloor)\dd},\mu P_{n\dd} \big)\\
	\le &C_3\big( d^{\ell_{\star\star}}+\mu(|\cdot|^{p_\star})\big)\delta+\W_2\big(\big(\mu P^{(\dd)}_{m\lfloor \dd^{-1}\rfloor\dd}\big)P_{(n-m\lfloor \dd^{-1}\rfloor)\dd},\mu P_{n\dd} \big),
\end{aligned}
\end{equation}
where we utilized   $0\le (n-m\lfloor \dd^{-1}\rfloor)\delta\le 1 $ due to $\lfloor \dd^{-1}\rfloor\dd\le 1. $
Furthermore, it is easy to see that
\begin{align*}
&\W_2\big(\big(\mu P^{(\dd)}_{m\lfloor \dd^{-1}\rfloor\dd}\big)P_{(n-m\lfloor \dd^{-1}\rfloor)\dd},\mu P_{n\dd} \big)\\
&\le \sum_{i=0}^{m-1}\mathbb W_2\big(\big(\mu P^{(\dd)}_{i\lfloor \dd^{-1}\rfloor\dd}\big)P_{ (n -i \lfloor \delta^{-1}\rfloor) \delta},\big(\mu P^{(\dd)}_{(i+1)\lfloor \dd^{-1}\rfloor\dd}\big)P_{(n- (i+1) \lfloor \delta^{-1}\rfloor) \delta}\big)\\
&=\sum_{i=0}^{m-1}\mathbb W_2\big(\big(\mu P^{(\dd)}_{i\lfloor \dd^{-1}\rfloor\dd}\big)P_{\lfloor \delta^{-1}\rfloor \delta}P_{ (n -(i+1) \lfloor \delta^{-1}\rfloor) \delta},\big(\mu P^{(\dd)}_{(i+1)\lfloor \dd^{-1}\rfloor\dd}\big)P_{(n- (i+1) \lfloor \delta^{-1}\rfloor) \delta}\big)\\
&\le C_1\sum_{i=0}^{m-1}\e^{-\ll_*(n- (i+1) \lfloor \delta^{-1}\rfloor) \delta} \W_2\big(\big(\mu P^{(\dd)}_{i\lfloor \dd^{-1}\rfloor\dd}\big)P_{ \lfloor \delta^{-1}\rfloor \delta}, \big(\mu P^{(\dd)}_{i\lfloor \dd^{-1}\rfloor\dd}\big) P^{(\dd)}_{ \lfloor \dd^{-1}\rfloor\dd} \big),
\end{align*}
where in the second identity we used the semigroup property, and the last display is valid by taking advantage of \eqref{EE5}. Next, by virtue of  $ \lfloor \delta^{-1}\rfloor \delta\le1$,
 applying Lemma \ref{lemma4} once more yields that for some constant $C_4>0$,
  \begin{align*}
\mathbb W_2\big(\big(\mu P^{(\dd)}_{i\lfloor \dd^{-1}\rfloor\dd}\big)P_{ \lfloor \delta^{-1}\rfloor \delta}, \big(\mu P^{(\dd)}_{i\lfloor \dd^{-1}\rfloor\dd}\big) P^{(\dd)}_{ \lfloor \dd^{-1}\rfloor\dd} \big)\le C_4\big(d^{\ell_{\star\star}}+\mu(|\cdot|^{p_\star})\big)\delta,\quad i=0,\cdots,m-1.
\end{align*}
Additionally, in the light of  $1/2\le\lfloor \delta^{-1}\rfloor \delta\le 1$ for $\delta\in(0,1/2]$, we derive that
\begin{align*}
\e^{-\ll_* n\dd}\sum_{i=1}^{m}\e^{\ll_* i\lfloor \delta^{-1}\rfloor \delta}
&\le \e^{-\ll_* n\dd} (\ll_*\lfloor \delta^{-1}\rfloor\delta)^{-1}\e^{\lambda_*  (m+1)  \lfloor \delta^{-1}\rfloor\delta }\le  \e^{\ll_*}(\ll_*/2)^{-1} ,
\end{align*}
where  in the first inequality we utilized  the basic inequality: $1/(\e^r-1)\le 1/r$ for $r>0$, and in the second inequality we took advantage of the fact that
 $0\le(n-m\lfloor \dd^{-1}\rfloor)\dd\le 1$ owing to  $\lfloor \delta^{-1}\rfloor \delta\le1$.
As a consequence, we conclude that there is a  dimension-free  constant $C_5>0$ such that for any $n\ge0$,
\begin{align}\label{RT2}
\W_2\big(\big(\mu P^{(\dd)}_{m\lfloor \dd^{-1}\rfloor\dd}\big)P_{(n-m\lfloor \dd^{-1}\rfloor)\dd},\mu P_{n\dd} \big)\le C_5\big( d^{\ell_{\star\star}}+\mu(|\cdot|^{p_\star})\big).
\end{align}
Afterward, \eqref{RT} follows by plugging \eqref{RT2} back into \eqref{RT1}.
\end{proof}

\section{Criteria on  Strong law of large numbers and proof of Theorem \ref{LLN0}}\label{sec5}

The following theorem establishes the strong LLN for \eqref{E2}, demonstrating that the additive functional (i.e., the time average) almost surely converges to the spatial average for $\pi_\infty$.
\begin{theorem}\label{LLN}
Assume   $({\bf A}_0)$, $({\bf A}_1)$ and  $({\bf H}_1)$-$({\bf H}_5)$, and suppose further $({\bf A}_5)$ once $\beta $ involved  in $({\bf H}_5)$
satisfies  $\beta> 1/2.$
Then, for any $f\in\mathscr C_\rho$ $($given in \eqref{f1}$)$, $\vv\in(0,1/2)$, and $x\in\R^d,$ there exist constants $C_0=C_0(x,\rho,\vv,\|f\|_\rho),\si_0>0$ and a random time $N_{\vv,\dd}=N_{\vv,\dd}(x,\gamma,d)\ge 1$  such that for   $\dd\in(0,\dd_2^{\star\star}]$, $n\ge N_{\vv,\dd}$, and the noise intensity $\si$ satisfying  $|\si|\ge\si_0,$
\begin{equation}\label{EE1}
\bigg|\frac{1}{n}\sum_{k=0}^{n-1} f(X_{k\dd}^{\dd,x})-\pi_\8(f) \bigg|\le C_0 \big(n^{-{1}/{2}+\vv}\dd^{-1/2}+\dd^{1\wedge \bb}{  d^{\frac{1}{2} (\ell_\star+  \rho) } }\big),\quad a.s.,
\end{equation}
where
$(X^{\dd,x}_{n\dd})_{n\ge 0}$  is determined by \eqref{E2},    $\pi_\8$ represents the IPM of $(X_{t})_{t\ge 0}$ solving \eqref{E1},  and  $\ell_\star$ was defined in \eqref{DD5}.
Moreover, for any $q>0,$ there is a constant $C_q^*>0$ such that
\begin{align}\label{7W}
	\E N_{\vv,\dd}^q\le C^*_q\big(1+|x|^{2(1+\rho)} +d^{1+\rho}\big)^{\frac{q+2}{2\vv}}.
\end{align}
\end{theorem}

\begin{proof}
In the following analysis, we stipulate $\delta\in(0,\dd_2^{\star\star}]$.
Set for  $n\ge1$ and $x\in\R^d$,
$$S^{\dd,x}_n:=\frac{1}{n}\sum_{k=0}^{n-1}   f(X_{k\dd}^{\dd,x})-\pi_\8^{(\dd)}(f),\quad f\in\mathscr C_\rho.$$ Under (${\bf H}_1$) and (${\bf H}_{2}$),  by  Theorem \ref{thm} and Lemma \ref{lemma3},
$(X_{n\delta}^{\delta,x})_{n\ge0}$ admits a unique  IPM $\pi^{(\dd)}_\8$. Obviously, we have
\begin{align*}
\bigg|\frac{1}{n}\sum_{k=0}^{n-1}   f(X_{k\dd}^{\dd,x})-\pi_\8(f)\bigg|\le |\pi_\8(f)-\pi^{(\dd)}_\8(f)|+|S^{\dd,x}_n|,
\end{align*}
where $\pi_\8$ is the IPM of $(X_t)_{t\ge0}$.
Consequently, the assertion \eqref{EE1} is verified  provided that there exists a constant $C_0>0$ such that
\begin{align}\label{1W}
 \big|\pi_\8(f)-\pi^{(\dd)}_\8(f)\big|\le C_0\delta^{1\wedge\beta}d^{\frac{1}{2} (\ell_\star +\rho) } ,
\end{align}
and that, for  $\vv\in(0,1/2)$, there is a random variable $N_{\vv,\dd}=N_{\vv,\dd}(x,\rho,d)$ so that for  all $n\ge N_{\vv,\dd}$,
\begin{align}\label{W1}
|S^{\dd, x}_{n}|\le n^{-\frac{1}{2}+\vv}\dd^{-\frac{1}{2}},\quad a.s.
\end{align}

Via H\"older's inequality and Minkowski's inequality, we find that
for    $f\in\mathscr C_\rho $ and  $\pi\in\mathscr C(\pi_\8, \pi^{(\dd)}_\8)$,
\begin{align*}
	\big|\pi_\8(f)-\pi^{(\dd)}_\8(f)\big|&\le\int_{\R^d\times\R^d}|f(x)-f(y)|\pi(\d x,\d y)\\
&\le\|f\|_\rho\int_{\R^d\times\R^d}| x -y |(1+|x|^\rho+|y|^\rho)\pi(\d x,\d y)\\
&\le \|f\|_\rho\bigg(\int_{\R^d\times\R^d}| x -y |^2\pi(\d x,\d y)\bigg)^{\frac{1}{2}}\big(1+\pi_\8(|\cdot|^{2\rho})^{\frac{1}{2}}+\pi^{(\dd)}_\8(|\cdot|^{2\rho})^{\frac{1}{2}}\big).
\end{align*}
Thus, taking infimum with respect to $\pi$ on both sides yields that
\begin{align}\label{2W}
\big|\pi_\8(f)-\pi^{(\dd)}_\8(f)\big|\le \|f\|_\rho\big(1+\pi_\8(|\cdot|^{2\rho})^{\frac{1}{2}}+\pi^{(\dd)}_\8(|\cdot|^{2\rho})^{\frac{1}{2}}\big)\mathbb W_2(\pi^{(\dd)}_\8,\pi_\8).
\end{align}
Next, under (${\bf H}_1$), (${\bf H}_2$) and (${\bf H}_4$), Lemmas \ref{lemma1} and \ref{lemma3} are applicable so that  there exists a constant $C_2>0$,
\begin{align*}
 \pi_\8(|\cdot|^{2\rho})^{\frac{1}{2}}+\pi^{(\dd)}_\8(|\cdot|^{2\rho})^{\frac{1}{2}}\le C_2\big(1+d^{\frac{1}{2} \rho }\big).
\end{align*}
Consequently, \eqref{1W} follows from Theorem \ref{IPM}.

Once we can claim   that, for integer $q\ge1$, there exists  a constant    $C_2=C_2(q)>0 $ such that
\begin{equation}\label{W3}
\mathbb{E}\big|S^{\dd, x}_n\big|^{2q}\le C_2  (1+|x|^{2(1+\rho)}+d^{1+\rho})^{q}(n\dd)^{-q},
\end{equation}
  Chebyshev's inequality and  H\"older's inequality imply that for some constant $C_3=C_3(\vv)>0$,
  \begin{align*}
\P\big(|S^{\dd,x}_{n}|>n^{-\frac{1}{2}+\vv}\dd^{-\frac{1}{2}}\big)
 &\le \dd^{\frac{1}{\vv}}n^{\frac{1}{\vv}-2}\big( \E|S^{\dd,x}_{n}|^{2\lceil1/\vv\rceil}\big)^{\ff{1/{\vv}}{\lceil1/\vv\rceil}}\\
  &\le C_3 (1+|x|^{2(1+\rho)}+d^{1+\rho})^{\frac{1}{\vv}}n^{-2}.
\end{align*}
 Accordingly, due to  $\sum_{n=1}^\8 n^{-2}  <\8$,  Borel-Cantelli's lemma yields that there exist a   random variable $N_{\vv,\dd}^*=N_{\vv,\dd}^*(x,\rho,d)>1$ such that for any $n\ge N_{\vv,\dd}^*$,
\begin{align*}
 |S_n^{\dd,x} |\le n^{-\ff12+\vv}\dd^{-\ff12}, \quad a.s.
\end{align*}
 Therefore,  \eqref{W1} is available by taking
\begin{align*}
N_{\vv,\dd}:=\min \big\{m\in \N: \big|S_n^{\dd,x}\big|\le n^{-\ff12+\vv}\dd^{-\ff12},~ \text{for any} ~n\ge m+1\big\}.
\end{align*}

 Owing to $N_{\vv,\dd}\le N^*_{\vv,\dd},$ we thus have $N_{\vv,\dd}<\8, a.s.$
 Moreover, in terms  of the definition of $N_{\vv,\dd}$ and Chebyshev's inequality, for any $q>0,$ it is easy to see that  there exists a constant $C_4=C_4(q)$ such that
 \begin{align*}
	\E N_{\vv,\dd}^q=\sum_{n=1}^\8\P(N_{\vv,\dd}=n)n^q
&\le \sum_{n=1}^\8 \P\big(|S^{\dd,x}_{n}|>n^{-\ff12+\vv}\dd^{-\ff12}\big) n^q\\
&\le C_4 \big(1+|x|^{2(1+\rho)} +d^{1+\rho}\big)^{\frac{q+2}{2\vv}}\sum_{n=1}^\8n^{-2}.
\end{align*}
As a result,  the assertion \eqref{7W} follows.

 It remains to prove \eqref{W3}.  By mimicking the strategy to derive \eqref{2W}, we deduce from Theorem \ref{thm} and Lemma \ref{lemma3} that for some constant  $C_5 >0$,
\begin{equation}\label{W4}
\begin{aligned}
\big|P^{(\dd)}_{n\dd}f(x)-\pi_\8^{(\dd)}(f)\big|&\le \|f\|_\rho\W_2\big(\dd_{x}P^{(\dd)}_{n\dd},\pi^{(\dd)}_\8\big)\Big(1+\big(\E|X_{n\dd}^{\dd,x}|^{2\rho}\big)^{\frac{1}{2}}+ \pi^{(\dd)}_\8(|\cdot|^{2\rho})^{\frac{1}{2}}\Big)\\
&\le\|f\|_\rho\big(|x|+\pi^{(\dd)}_\8(|\cdot|^2)^{\frac{1}{2}}\big)\Big(1+\big(\E|X_{n\dd}^{\dd,x}|^{2\rho}\big)^{\frac{1}{2}}+ \pi^{(\dd)}_\8(|\cdot|^{2\rho})^{\frac{1}{2}}\Big)\\
&\le C_5\|f\|_\rho\e^{-\ll n\dd}\big(1+|x|^{ 1+\rho }+d^{\frac{1}{2}(1+\rho) }\big),
\end{aligned}
\end{equation}
where in the second inequality Young's inequality was applied.
Furthermore, due to $f\in\mathscr C_\rho,$ it is easy to see that for any $x\in\R^d,$
\begin{align}\label{3W}
|f(x)|\le |f({\bf0})|+|f(x)-f({\bf0})|\le 2\big(|f({\bf0})|\vee\|f\|_\rho\big) (1+ |x|^{1+\rho}).
\end{align}
Next, for any integer $q\ge1$,
a direct calculation shows that for some constant $C_6=C_6(q)$,
\begin{align*}
 \Big(\E\big|S^{\dd,x}_{n}\big|^{2q}\Big)^{\frac{1}{q }}
 &\le  \ff{C_6}{n^2} \sum^{n}_{m=1}\sum^{1\vee(m-1)}_{k=1}
\Big( \E\big|\big(f(X^{\dd,x}_{(k-1)\dd})-\pi_\8^{(\dd)}(f)\big) \big(P_{(m-k)\delta}^{(\delta)}f(X^{\dd,x}_{(k-1)\dd})-\pi_\8^{(\dd)}(f)  \big)\big|^{q}\Big)^{\frac{1} { q }}.
\end{align*}
This, together with \eqref{W4} and \eqref{3W}, gives that for some constant $C_7>0,$
\begin{align*}
\Big(\E\big|S^{\dd,x}_{n\dd}\big|^{2q}\Big)^{\frac{1}{ q}} &\le  \ff{C_7}{n^2} \sum^{n}_{m=1}\sum^{1\vee(m-1)}_{k=1} \e^{-\ll (m-k)\dd}\Big(1+d^{1+\rho}
+\big(\E|X_{(k-1)\delta}^{\delta,x}|^{2q(1+\rho)}\big)^{\frac{1} { q}}\Big)\\
&\le \ff{C_8}{n\dd} \big(1+|x|^{2(1+\rho)}+d^{1+\rho}\big),
\end{align*}
where  we exploited Lemma \ref{lemma3} and the fact that
 $\sum^{n}_{m=1}\sum^{m-1}_{k=1} \e^{-\ll (m-k)\dd} \le n(\ll \dd)^{-1}  $  in the second inequality.  Subsequently, \eqref{W3} follows directly.
\end{proof}

At the end of this section, we finish the proof of Theorem \ref{LLN0}.

\begin{proof}[Proof of Theorem \ref{LLN0}]
In terms of Theorem \ref{LLN}, the proof of  Theorem \ref{LLN0}  can be done by noting that
all conditions required in Theorem \ref{LLN} have   been examined in the proof of Theorem \ref{cor}.
\end{proof}

\noindent {\bf Acknowledgements.}\,\,
We would like to   thank 
two referees for their  constructive
comments and suggestions, which have improved our manuscript considerably.
The research of Jianhai Bao is supported by the National Key R\&D Program of China (2022YFA1006004).

\end{document}